\documentclass[a4paper,12pt]{article}
\usepackage{geometry}
\usepackage{mathtools}
\usepackage{amssymb}
\usepackage{xypic}
\usepackage[all,2cell]{xy}
\usepackage{graphics}
\usepackage{amsmath}
\usepackage{fullpage}
\usepackage{theorem}
\usepackage{amstext}
\usepackage{mathrsfs}
\usepackage{calligra}
\usepackage{xcolor, soul}
\sethlcolor{yellow}

\newtheorem{theorem}{Theorem}[section]
\newtheorem{corollary}{Corollary}[section]
\newtheorem{proposition}{Proposition}[section]

\newtheorem{cex}{Counter Example}
\newtheorem{lemma}{Lemma}[section]
\theorembodyfont{\upshape}

\newtheorem{definition}{Definition}[section]

\newtheorem{remark}{Remark}[section]

\newenvironment{proof}[1][Proof]{\textbf{#1.} }{\ \rule{0.5em}{0.5em}}
\begin{document}
	
	\title{A groupoid approach to the study of fuzzy topological spaces}
	\author{
		Anjeza Krakulli \\
		Universiteti Aleksand\"er Moisiu \\
		Fakulteti i Teknologjis\"e dhe Informacionit \\
		Departamenti i Matematik\"es, Durr\"es \\
		anjezakrakulli@uamd.edu.al \\
		Elton Pasku \\
		Universiteti i Tiran\"es \\
		Fakulteti i Shkencave Natyrore \\
		Departamenti i Matematik\"es, Tiran\"e \\
		elton.pasku@fshn.edu.al \\
	}
	
	\date{}
	
	\maketitle
	
\begin{abstract}
The definition of the complement of a fuzzy subset is algebraic in nature and when it is used in the context of fuzzy topological spaces it does not share any similarity with the usual property of topological spaces that the complement of an open subset is closed. To tackle this inconsistency, we associate to any fuzzy topological space a topological space and use its fundamental groupoid equipped with the Lasso topology to give a topological characterization for the complementation of fuzzy subsets.
\textbf{AMS Mathematics  Subject Classification $(2020)$}: 22A22, 55U40, 03E72, 54A40.
\end{abstract}

	\section{Introduction}
	
	The following is the fuzzy version of the definition of a topological space and appears in \cite{FTS}.
\begin{definition} \label{fts}
A fuzzy topology is a family $\mathcal{T}$ of fuzzy sets in $X$ which
	satisfies the following conditions:
	\begin{itemize}
		\item [(i)] $\emptyset, X \in \mathcal{T}$,
		\item[(ii)] If $A,B \in \mathcal{T}$, then $A \wedge B \in \mathcal{T}$,
		\item[(iii)] For every family $I$ and every $A_{i} \in \mathcal{T}$ with $i \in I$, $\underset{i \in I}{\vee} A_{i} \in \mathcal{T}$.
	\end{itemize}
\end{definition}
This definition generalizes that of a topological space since every topological space $(X, \tau)$ can be regarded as a fuzzy topological space in the above sense. Indeed, for every $T \in \tau$, the corresponding indicator map $I_{T}$ on $X$ is a fuzzy subset of $X$ in the sense of \cite{Z}, and the collection $\mathcal{T}$ of all such $I_{T}$ is a fuzzy topology on $X$. The indicator map $I_{\theta}$ of the empty map $\theta$ is the constant map $\emptyset$ at zero, the indicator map $I_{X}$ is the constant map $X$ at 1, and conditions (ii) and (iii) of definition \ref{fts} are easily verified. Of course there are examples of fuzzy topological spaces which do not arise from topological spaces in the above manner. One such example has been provided by Lowen in \cite{Lowen} who explores some interesting relationships between topological and fuzzy topological spaces. Lowen associates to any fuzzy topological space $(X, \mathcal{T})$ the topological space $(X, \iota(\mathcal{T}))$, where $\iota(\mathcal{T})$ is the initial topology on $X$ for the family of maps in $\mathcal{T}$ and the topological space $I_{r}$. In the reverse direction, he associates to any topological space $(X ,\tau)$ the fuzzy topological space $(X,\mathcal{C}(X, I_{r}))$ where $\mathcal{C}(X, I_{r})$ is the set of all continuous maps from $(X ,\tau)$ to $I_{r}$. It is obvious that the maps of $\mathcal{C}(X, I_{r})$ are not indicator maps in general, therefore $(X,\mathcal{C}(X, I_{r}))$ serves as a counterexample making thus the difference between topological and fuzzy topological spaces apparent. It should be mentioned that there is another, not so obvious reason, which distinguishes the two notions and has to do with the definition of the complement of a fuzzy subset $F$ of a set $X$ as being the map $1-F$. This definition, which is used widely in the context of fuzzy topological spaces too, is associated with the following pathology. Given any fuzzy topological space $(X, \mathcal{T})$, $T \in \mathcal{T}$ a fuzzy open subset of $X$ and $T^{-1}(\gamma,1]$ any sub bases member for the topological space $(X, \iota(\mathcal{T}))$, the complement of $T^{-1}(\gamma,1]$ in $(X, \iota(\mathcal{T}))$ is in general different than $(1-T)^{-1}(\gamma,1]$ even if $1-T$ is fuzzy open as the following example shows. Let for instance $X$ be any nonempty set and $\mathcal{T}=\{1/3,2/3,1, \emptyset\}$ consisting of four constant maps. If $\gamma=0$ and $T=1/3$, then $1-T=2/3 \in \mathcal{T}$ and $T^{-1}(\gamma, 1]=X=(1-T)^{-1}(\gamma, 1]$, but $X \setminus (T^{-1}(\gamma, 1])=\theta \neq (1-T)^{-1}(\gamma, 1]$. It is now natural to ask the following question. If fuzzy complementary subsets $F$ and $1-F$ in $(X, \mathcal{T})$ do not give rise in general to set theoretic complementary subsets in $(X, \iota(\mathcal{T}))$, then is there any other space related to $(X, \mathcal{T})$ in which we can interpret in algebraic or topological terms the process of taking complements of fuzzy subsets in $(X, \mathcal{T})$? The main scope of this paper is to address this question. We leave the details for the definition of that space for the next section, but we have to mention here that our approach is heavily based on the notion of the topological groupoid $\pi^{l}Y$ of a given space $Y$ defined from Pakdaman and Shahini in \cite{LT}. If $Y$ is a topological space, which is assumed to be connected and locally path connected and $\pi Y$ its fundamental groupoid, Pakdaman and Shahini defined a topology on $\pi Y$ provided that a bases for $Y$ is known. With this topology, the groupoid $\pi Y$ now becomes a topological groupoid and is denoted throughout by $\pi^{l}Y$. Pakdaman and Shahini definition is a generalization of the Lasso topology defined in \cite{Br} which makes the fundamental group of a space a topological group. It is shown in \cite{LT} that a bases for $\pi^{l}Y$ ''modulo'' a given bases $\mathcal{U}$ of $Y$ consists of sets $N([\alpha], \mathcal{U}, V, W)$ with $[\alpha] \in \pi(Y, x,y)$ where $V,W \in \mathcal{U}$ are open nighborhoods of $x, y$ respectively, which have elements $[\beta]$ such that
$$\beta \sim \gamma \ast \mu \ast \alpha \ast \mu' \ast \lambda,$$
where $\gamma$ is a path in $V$ with $\gamma(1)=x$, $\mu \in \pi_{1}(\mathcal{U}, x)$, $\mu' \in \pi_{1}(\mathcal{U}, y)$ and $\lambda$ is a path in $W$ such that $\lambda(0)=y$. Here $\pi_{1}(\mathcal{U}, x)$ is the so called the Spanier group with respect to $\mathcal{U}$ and is by definition the subgroup of $\pi_{1}(Y, x)$ consisting of homotopy classes of loops that can be represented by a finite product of the form $\underset{i \in I}{\prod}u_{i}v_{i}u^{-1}_{i}$ where $u_{i}(0)=x$ and $v_{i}$ is a loop which sits inside some $U_{i} \in \mathcal{U}$. We conclude this section by stating that everything which is used in the paper but is not explained in the introduction or latter, can be found in \cite{Brown} and in \cite{LT} if it is related to groupoids, and in \cite{Rotman-T} if it is related to Algebraic Topology. Necessary notions from fuzzy subsets and fuzzy topological spaces are in \cite{FTS}, \cite{Lowen} and \cite{Z}.

	\section{From fuzzy topological spaces to topological spaces}\label{rt}

	Given a fuzzy topology $\mathcal{T}$ on a set $X$, we will define a topology $\Psi^{\ast}(\mathcal{T})$ on the Cartesian product $X \times J$ where $J=[0,1)$ as follows. It is more convenient to define first a map
	$$\Psi^{\ast}: [0,1]^{X} \rightarrow \mathcal{P}(X \times J)$$
	such that for every $T \in [0,1]^{X}$,
	$$\Psi^{\ast}(T)=\{(x,\alpha) \in X \times J| T(x) > \alpha\}.$$
We denote for further use $X_{T}^{\ast}=\{x \in X| T(x) \neq 0\}$ and note that $\pi_{1}(\Psi^{\ast}(T))=X_{T}^{\ast}$, where $\pi_{1}$ is the projection in the first coordinate. 
In the extreme case when $T=\emptyset$ we see that
$$
(x,\alpha) \in \Psi^{\ast}(\emptyset)  \Leftrightarrow 0=\emptyset(x) > \alpha
$$
which is impossible, therefore $ \Psi^{\ast}(\emptyset)=\emptyset$. On the other hand, it is clear that for $T \neq \emptyset$, $\Psi^{\ast}(T)$ can be expressed as a disjoint union of sets as follows
\begin{equation} \label{du}
\Psi^{\ast}(T)=\underset{x \in X_{T}^{\ast}}{\bigcup}\{x\} \times [0,T(x)),
\end{equation}
and that for every fixed $x \in X_{T}^{\ast}$,
	\begin{equation} \label{sup}
		T(x)=\text{sup}\{\alpha \in J| (x,\alpha) \in \Psi^{\ast}(T)\}.
	\end{equation}
Let us show that if we restrict $\Psi^{\ast}$ on $\mathcal{T}$, then the respective family of subsets $\Psi^{\ast}(\mathcal{T})=\{\Psi^{\ast}(T)| T \in \mathcal{T}\}$ constitutes to a topology on $X \times J$. Indeed, as we previously observed, $\Psi^{\ast}(\emptyset)=\emptyset$, and so $\emptyset \in \Psi^{\ast}(\mathcal{T})$. If we take $T=X$, then
	\begin{align*}
		(x,\alpha) \in \Psi^{\ast}(X)  &\Leftrightarrow 1=X(x) > \alpha\\
		& \Leftrightarrow (x,\alpha) \in X \times J,
	\end{align*}
	which shows that $X \times J \in \Psi^{\ast}(\mathcal{T})$.
	Now let $T_{1}, T_{2} \in \mathcal{T}$. We show that
\begin{equation} \label{pinf}
\Psi^{\ast}(T_{1}) \cap \Psi^{\ast}(T_{2})=\Psi^{\ast}(T_{1} \wedge T_{2}) 
\end{equation}
from which we get that $\Psi^{\ast}(T_{1} \wedge T_{2}) \in \Psi^{\ast}(\mathcal{T})$ since $T_{1} \wedge T_{2} \in \mathcal{T}$. This proves that $\Psi^{\ast}(\mathcal{T})$ is closed under finite intersections. Indeed, if $T_{1}\wedge T_{2}=\emptyset$, then there are no pairs $(x,\alpha) \in   \Psi^{\ast}(T_{1}) \cap \Psi^{\ast}(T_{2})$, therefore 
	$$\Psi^{\ast}(T_{1}) \cap \Psi^{\ast}(T_{2})=\emptyset=\Psi^{\ast}(T_{1} \wedge T_{2}). $$
If $T_{1} \wedge T_{2} \neq \emptyset$, then for some $x \in X$, $T_{1}(x) \wedge T_{2}(x) \neq 0$, consequently $\Psi^{\ast}(T_{1}) \cap \Psi^{\ast}(T_{2}) \neq \emptyset$ since for every $0 <\alpha < T_{1}(x) \wedge T_{2}(x)$ we have that $(x,\alpha) \in \Psi^{\ast}(T_{1}) \cap \Psi^{\ast}(T_{2})$. Further, to complete the proof for (\ref{pinf}) in the latter case, we see that
	\begin{align*}
		(x,\alpha) \in \Psi^{\ast}(T_{1}) \cap \Psi^{\ast}(T_{2}) & \Leftrightarrow T_{1}(x) > \alpha \text{ and } T_{2}(x)> \alpha\\
		& \Leftrightarrow (T_{1}(x) \wedge T_{2}(x))>\alpha\\
		& \Leftrightarrow (T_{1} \wedge T_{2})(x) > \alpha\\
		& \Leftrightarrow (x,\alpha) \in \Psi^{\ast}(T_{1} \wedge T_{2}).
	\end{align*}
Finally, let $I$ be any nonempty index set and let $T_{i} \in \mathcal{T}$ for $i \in I$. We show that 
\begin{equation} \label{psup}
\underset{i \in I}{\cup} \Psi^{\ast}(T_{i})=\Psi^{\ast}(\underset{i \in I}{\vee}(T_{i})),
\end{equation}
from which we get that $\Psi^{\ast}(\underset{i \in I}{\vee}(T_{i}))\in \Psi^{\ast}(\mathcal{T})$ since $\underset{i \in I}{\vee}T_{i} \in \mathcal{T}$. This proves that $\Psi^{\ast}(\mathcal{T})$ is closed under arbitrary unions. Indeed, if each $T_{i} =\emptyset$, then the equality is clear. Otherwise, $\underset{i \in I'}{\cup} \Psi^{\ast}(T_{i})$ is nonempty in which case we see that
	\begin{align*}
		(x,\alpha) \in \underset{i \in I}{\cup} \Psi^{\ast}(T_{i}) & \Leftrightarrow \exists i \in I, (x,\alpha) \in \Psi^{\ast}(T_{i})\\
		& \Leftrightarrow \exists i \in I, T_{i}(x) > \alpha\\
		& \Leftrightarrow \underset{i \in I}{\vee}T_{i}(x) > \alpha\\
		& \Leftrightarrow (x,\alpha) \in \Psi^{\ast}(\underset{i \in I}{\vee}T_{i}).
	\end{align*}
The problem with the topology $\Psi^{\ast}(\mathcal{T})$ on $X \times J$ defined in terms of the fuzzy topology $\mathcal{T}$ on $X$, is that the correspondence $\Psi^{\ast}$ does not extends to closed subsets. More specifically, it may happen that for some fuzzy open subset $T \in \mathcal{T}$, we have that $\Psi^{\ast}(1-T) \neq (X \times J) \setminus \Psi^{\ast}(T)$ as the subsequent example shows. In contrast with this, when $T$ is the indicator map $I_{A}$ of some $A \subseteq X$, we have that $\Psi^{\ast}(1-T) = (X \times J) \setminus \Psi^{\ast}(T)$. Indeed, when $A=X$, then $I_{A}=1$ the constant map at 1 and $1-I_{A}=\theta$ the constant map at zero. On the other hand $\Psi^{\ast}(1)=X \times J$ and $\Psi^{\ast}(\theta)=\emptyset$. In this case we see that $\Psi^{\ast}(1-I_{A}) = (X \times J) \setminus \Psi^{\ast}(I_{A})$. The case when $A=\emptyset$ is dual to the above. When $A \subset X$ is a nontrivial subset, we observe that $1-I_{A}=I_{X \setminus A}$ and that 
\begin{align*}
	\Psi^{\ast}(1-I_{A}) & = \{(x,\alpha) \in X \times J: I_{X \setminus A}(x) > \alpha\}\\
	&= (X \setminus A) \times J\\
	&= (X \times J) \setminus (A \times J)\\
	&= (X \times J) \setminus \{(x,\alpha) \in X \times J: I_{A}(x)> \alpha\}\\
	&=(X \times J) \setminus \Psi^{\ast}(I_{A}).
\end{align*}
\begin{cex}
	Let $X$ be any nonempty set and let $\mathcal{T}=\{\theta, 1, \dfrac{1}{3}\}$ be the fuzzy topology on $X$ consisting of only these three constant maps. Letting $T=\dfrac{1}{3}$ we see that
	$$\Psi^{\ast}(T)=\{(x,\alpha) \in X \times J: \dfrac{1}{3}>\alpha\}=X \times [0, 1/3[,$$
	and then
	$$(X \times J) \setminus \Psi^{\ast}(T)= X \times [1/3,1[.$$
	But 
	$$\Psi^{\ast}(1-T)=\{(x,\alpha) \in X \times J: 1-1/3> \alpha\}=X \times [0,2/3[,$$
	which shows that $\Psi^{\ast}(1-T) \neq (X \times J) \setminus \Psi^{\ast}(T)$.
\end{cex}
This oddity that steams from the definition of fuzzy complements which is algebraic in nature, motivates us to construct another, more complex topological space in which taking complements of fuzzy subsets on the one side, is equivalent on the other side with taking inverses. This idea has led us to consider the groupoid of homotopy classes of paths on $X \times J$ but the topology $\Psi^{\ast}(\mathcal{T})$ is to coarse to achieve anything, therefore it is tempting to refine $\Psi^{\ast}(\mathcal{T})$ in a reasonable way. Our approach in doing this is based on Lowen's idea in \cite{Lowen} who associates any given fuzzy topology $\mathcal{T}$ on $X$ with the initial topology on $X$ for the family of maps $T\in \mathcal{T}$ and the topological space $I_{r}$. 
\begin{definition}
	Consider the set $(-1, 1]$ and the family of all left intervals $(\gamma,1]$ with $-1 \leq \gamma <1$. We let $(-1, 1]_{r}$ be the topological space generated by this family.
\end{definition}

\begin{definition}
	For every $T \in \mathcal{T}$ we define $T^{\ast}: X \times J \rightarrow (-1, 1]_{r}$ by $T^{\ast}(x,\alpha)=T(x)-\alpha$. 
\end{definition}

\begin{definition}
	Let $\pi_{2}: X \times J \rightarrow (-1, 1]_{r}$ be the projection in the second coordinate, $\pi_{2}(x,\alpha)=\alpha$.
\end{definition}

\begin{definition}
	Let $\iota(\mathcal{T})$ be the initial topology on $X \times J$ for the family of maps $T^{\ast}$ with $T \in \mathcal{T}$ together with $\pi_{2}$, and the topological space $(-1, 1]_{r}$.
\end{definition}

\begin{lemma}
	The topology $\iota(\mathcal{T})$ is finer than $\Psi^{\ast}(\mathcal{T})$.
\end{lemma}
\begin{proof}
	For every $T \in \mathcal{T}$, we see that $(T^{\ast})^{-1}((0,1])=\{(x,\alpha) \in X \times J: T(x)-\alpha \in (0,1]\}=\{(x,\alpha) \in X \times J: T(x)>\alpha \}=\Psi^{\ast}(T)$.
\end{proof}

\begin{definition}
Let $\iota_{X}(\mathcal{T})$ be the initial topology on $X$ for the family of maps $T \in \mathcal{T}$ and the topological space $(-1, 1]_{r}$. 
\end{definition}

\begin{lemma}
Every $T \in \mathcal{T}$ is a continuous map from the space $(X, \iota_{X}(\mathcal{T}))$ to $(-1, 1]_{r}$.
\end{lemma}
\begin{proof}
This follows from the general result that the initial topology on a family of maps makes these maps continuous.
\end{proof}

\begin{remark}
	A sub bases for $\iota(\mathcal{T})$ consists in all $(T^{\ast})^{-1}((\gamma,1])$ with $T \in \mathcal{T}$ and $-1 \leq \gamma <1$ together with $(\pi_{2})^{-1}((\gamma,1])$ with $-1 \leq \gamma <1$. To make the second more explicit, we see that $(\pi_{2})^{-1}((\gamma,1])=\{(x,\alpha) \in X \times J: \alpha > \gamma\}$. If we consider now the subspace $X \times \{0\}$ of $(X \times J, \iota(\mathcal{T}))$ where the topology is the one inherited from $\iota(\mathcal{T})$, we see that for every $\gamma < 0$, $(\pi_{2})^{-1}((\gamma,1])=\{(x,0) \in X \times \{0\}: 0 > \gamma\}=X \times \{0\}$, and for every $\gamma \geq 0$, $(\pi_{2})^{-1}((\gamma,1])=\{(x,0) \in X \times \{0\}: 0 > \gamma\}=\emptyset$. Therefore the only sub bases member for $(X \times \{0\}, \iota(\mathcal{T}))$ arising from $\pi_{2}$ is $X \times \{0\}$ itself. Finally, we mention that it follows from the definition $\iota_{X}(\mathcal{T})$ that a sub bases for $\iota_{X}(\mathcal{T})$ consist of all subsets $T^{-1}((\gamma, 1])$ with $T \in \mathcal{T}$ and $-1\leq \gamma < 1$. 
\end{remark}

\begin{lemma} \label{hmm}
	The projection in the first coordinate $\pi_{1}: X \times \{0\} \rightarrow X$ is a homeomorphism of spaces.
\end{lemma}
\begin{proof}
Let $(x,0) \in X \times \{0\}$ and $T^{-1}(\gamma,1]$ an open sub bases member containing $x=\pi_{1}(x,0)$. It follows that $T(x)> \gamma$, whence $(x,0) \in (T^{\ast})^{-1}(\gamma, 1]$. For every $(y, 0) \in (T^{\ast})^{-1}(\gamma, 1]$, since $T(y)> \gamma$, we have that $y \in T^{-1}(\gamma,1]$ which proves that $\pi_{1}$ is continuous. Conversely, let $(x,0)=\pi_{1}^{-1}(x)$ and let $(T^{\ast})^{-1}(\gamma, 1]$ be an open sub bases neighbourhood of $(x,0)$. Since $T(x)> \gamma$, we have that $x \in T^{-1}(\gamma,1]$. For every $y \in T^{-1}(\gamma,1]$, $T(y)> \gamma$, hence $\pi_{1}^{-1}(y)=(y,0) \in (T^{\ast})^{-1}(\gamma, 1]$ proving that $\pi_{1}^{-1}$ is also continuous.
\end{proof}

\begin{theorem} \label{defret}
	The space $(X \times J, \iota(\mathcal{T}))$ deformation retracts to $(X \times \{0\}, \iota(\mathcal{T}))$.
\end{theorem}
\begin{proof}
	Let $H: I \times (X \times J) \rightarrow X \times J$ defined by $H(t,(x,\alpha))=(x, (1-t)\alpha)$. First we note that $H(t,(x,0))=(x,0)$ for every $(x,0) \in X \times \{0\}$, $H(0,(x,\alpha))=(x,\alpha)$ for all $(x,\alpha) \in X \times J$ and $H(1,(x,\alpha))=(x,0)$ for all $(x,\alpha) \in X \times J$. Now we prove that $H$ is continuous by testing it on the sub bases described in the previous remark. We distinguish between the following cases for $t$.\\
	(i) $t=1$. Let $\pi^{-1}_{2}((\gamma,1])$ with $-1\leq \gamma <1$ be an open neighbourhood of $(x,0)=(x,(1-1)\alpha)$. It follows that $\gamma <0$. If $\alpha=0$, then for some appropriately chosen $\varepsilon>0$, $(1-\varepsilon,1] \times \pi^{-1}_{2}((\gamma,1])$ is an open neighbourhood of $(1,(x,0))$. For every $(1-\delta, (y, \beta)) \in (1-\varepsilon,1] \times \pi^{-1}_{2}((\gamma,1])$ we have that $H((1-\delta, (y, \beta)))=(y, \delta \beta) \in \pi^{-1}_{2}((\gamma,1])$ since $\delta \beta \geq 0> \gamma$. In the case when $\alpha > 0$, we can chose $\varepsilon>0$ such that $1-\varepsilon>0$ and $\alpha -\varepsilon>0$. It follows that an open neighbourhood of $(1,(x,\alpha))$ is $(1-\varepsilon,1] \times \pi^{-1}_{2}((\alpha -\varepsilon,1])$. For every $(1-\delta, (y,\beta)) \in (1-\varepsilon,1] \times \pi^{-1}_{2}((\alpha -\varepsilon,1])$ we have that $H((1-\delta, (y,\beta)))=(y, \delta \beta) \in \pi^{-1}_{2}((\gamma,1])$ since $\delta \beta\geq  \delta (\alpha -\varepsilon)\geq 0> \gamma$. In the case when $(T^{\ast})^{-1}((\gamma,1])$ with $-1\leq \gamma <1$ is an open neighbourhood of $(x,0)=(x,(1-1)\alpha)$, we have that $T(x)>\gamma$. If $\alpha=0$, then an open neighbourhood of $(1,(x,0))$ is $(1-\varepsilon,1] \times (T^{\ast})^{-1}((\gamma,1])$ where $0<1-\varepsilon< 1$. For every $(1-\delta, (y,\beta)) \in (1-\varepsilon,1] \times (T^{\ast})^{-1}((\gamma,1])$ we have that $H(1-\delta, (y,\beta))=(y, \delta \beta) \in (T^{\ast})^{-1}((\gamma,1])$. Indeed, 
	\begin{align*}
		T(y)&> \gamma +\beta && (\text{from the assumption that } (y,\beta) \in (T^{\ast})^{-1}((\gamma,1]))\\
		&\geq \gamma+ \delta \beta && (\text{since } \delta <\varepsilon),
	\end{align*}
	which means that $T(y)-\delta \beta \in (\gamma,1]$, or equivalently that $H(1-\delta, (y,\beta))=(y, \delta \beta) \in (T^{\ast})^{-1}((\gamma,1])$. In case $\alpha > 0$, we can chose $0< \varepsilon<1$ such that $T(x)-\alpha> \gamma -(\alpha -2 \varepsilon)$ and $\alpha -\varepsilon>0$. An open neighbourhood of $(1,(x,\alpha))$ now is $(1-\varepsilon,1] \times ((T^{\ast})^{-1}((\gamma -(\alpha -2 \varepsilon)] \cap \pi^{-1}_{2}(\alpha -\varepsilon,1])$. Let $(1-\delta, (y,\beta)) \in (1-\varepsilon,1] \times ((T^{\ast})^{-1}((\gamma -(\alpha -2 \varepsilon)] \cap \pi^{-1}_{2}(\alpha -\varepsilon,1])$, then $T(y)-\beta> \gamma -(\alpha -2 \varepsilon)$. It follows that for every $1-\delta \in (1-\varepsilon,1]$ we have
	\begin{align*}
		T(y)-(1-(1-\delta))\beta&= T(y)-\beta +(1-\delta)\beta\\
		&> \gamma -(\alpha -2 \varepsilon)+(1-\delta)(\alpha -\varepsilon)\\
		&= \gamma - (\alpha -\varepsilon) + \varepsilon +(\alpha -\varepsilon)-\delta (\alpha -\varepsilon)\\
		&= \gamma + \varepsilon -\delta (\alpha -\varepsilon)\\
		&\geq \gamma && (\text{since } \varepsilon -\delta (\alpha -\varepsilon)\geq 0)
	\end{align*}
	(ii) $0< t <1$. As in the previous case, an open neighbourhood of $(x, (1-t)\alpha)$ can be of two kinds. The first one is $\pi^{-1}_{2}((\gamma,1])$ with $-1\leq \gamma <1$. Then, $\alpha> \gamma +t \alpha$. We can find $\varepsilon> 0$ such that $(t-\varepsilon, t+\varepsilon) \subset [0,1]$ and $\alpha> \gamma +(t+\varepsilon) \alpha$. It follows that $(x,\alpha) \in \pi^{-1}_{2}((\gamma +(t+\varepsilon) \alpha,1])$, and that an open neighbourhood of $(t, (x,\alpha))$ is $(t-\varepsilon, t+\varepsilon)\times \pi^{-1}_{2}((\gamma +(t+\varepsilon) \alpha,1])$. We prove that $H\left((t-\varepsilon, t+\varepsilon)\times \pi^{-1}_{2}((\gamma +(t+\varepsilon) \alpha,1])\right) \subseteq \pi^{-1}_{2}((\gamma,1])$. Let $(t+\delta, (y,\beta)) \in (t-\varepsilon, t+\varepsilon)\times \pi^{-1}_{2}((\gamma +(t+\varepsilon) \alpha,1])$, hence $|\delta|< \varepsilon$ and $\beta> \gamma +(t+\varepsilon) \alpha$. For $H(t+\delta, (y,\beta))=(y, (1-(t+\delta))\beta)$ we see that
	\begin{align*}
		(1-(t+\delta))\beta & \geq (1-(t+\varepsilon))\beta\\
		&> (1-(t+\varepsilon))(\gamma +(t+\varepsilon)\alpha)\\
		&= \gamma + (t+\varepsilon)(\alpha -\gamma -(t+\varepsilon)\alpha)\\
		& \geq \gamma && (\text{since } \alpha >\gamma +(t+\varepsilon)\alpha).
	\end{align*}
	This proves that $H(t+\delta, (y,\beta)) \in \pi^{-1}_{2}((\gamma,1])$. The second kind of an open neighbourhood of $(x, (1-t)\alpha)$ is $(T^{\ast})^{-1}((\gamma,1])$ with $-1\leq \gamma <1$. In this case $T(x)-\alpha> \gamma -\alpha t$. We can chose $\varepsilon>0$ such that 
	$(t-\varepsilon, t+\varepsilon) \subset[0,1]$ and that $$T(x)-\alpha>  \gamma -\alpha t+(\varepsilon t+\varepsilon)=\gamma -(\alpha-(\varepsilon+\dfrac{\varepsilon}{t})) t.$$ 
Let $\mu=\text{max}\{-1, \gamma -(\alpha-(\varepsilon+\dfrac{\varepsilon}{t})) t\}$.
	It follows that $(x,\alpha) \in (T^{\ast})^{-1}((\mu,1]) \cap \pi^{-1}_{2}((\alpha -\varepsilon,1])$. Thus an open neighborhood of $(t,(x,\alpha))$ is $(t-\varepsilon, t+\varepsilon) \times ((T^{\ast})^{-1}((\mu,1]) \cap \pi^{-1}_{2}((\alpha -\varepsilon,1]))$. We prove that $H((t-\varepsilon, t+\varepsilon) \times ((T^{\ast})^{-1}((\mu,1]) \cap \pi^{-1}_{2}((\alpha -\varepsilon,1]))) \subseteq (T^{\ast})^{-1}((\gamma,1])$. Indeed, for every $(t+\delta, (y,\beta)) \in (t-\varepsilon, t+\varepsilon) \times ((T^{\ast})^{-1}((\mu,1]) \cap \pi^{-1}_{2}((\alpha -\varepsilon,1]))$, we have that $|\delta|< \varepsilon$, $T(y)-\beta>\mu\ \geq \gamma -(\alpha -(\varepsilon +\dfrac{\varepsilon}{t}))t$ and $\beta> \alpha -\varepsilon$. It follows that 
	\begin{align*}
		T(y)-(1-(t+\delta))\beta & =T(y)-\beta +t \beta + \delta \beta\\
		& > \gamma -(\alpha -(\varepsilon +\dfrac{\varepsilon}{t}))t +t (\alpha -\varepsilon)+  \delta \beta\\
		&=\gamma -(\alpha -\varepsilon)t +\varepsilon +t (\alpha -\varepsilon)+  \delta \beta\\
		&=\gamma +\varepsilon+\delta \beta\\
		&\geq \gamma && (\text{since } \varepsilon>\beta \varepsilon> \beta |\delta |,)
	\end{align*}
	which show that $H((t+\delta, (y,\beta))) \in (T^{\ast})^{-1}((\gamma,1])$.\\
(iii) $t=0$. Let $\pi_{2}^{-1}((\gamma,1])$ be an open neighbourhood of $H(0,(x,\alpha))=(x,\alpha)$, hence $\alpha> \gamma$. When $\alpha\neq 0$, we can find $0<\varepsilon<1$ such that $\alpha > \dfrac{1}{1-\varepsilon}\gamma$. We can in fact chose $0< \varepsilon <1-\dfrac{\gamma}{\alpha}$. It is clear that $(0,(x,\alpha)) \in [0,\varepsilon) \times \pi_{2}^{-1}((\dfrac{1}{1-\varepsilon}\gamma,1])$. Let now $(\delta, (y,\beta)) \in [0,\varepsilon) \times \pi_{2}^{-1}((\dfrac{1}{1-\varepsilon}\gamma,1])$, hence $\delta< \varepsilon$ and $\beta> \dfrac{1}{1-\varepsilon}\gamma$. It follows that $H((\delta, (y,\beta)))=(y, (1-\delta)\beta) \in \pi_{2}^{-1}((\gamma,1])$ since
	$$
	\beta> \dfrac{1}{1-\varepsilon}\gamma\geq  \dfrac{1}{1-\delta}\gamma,
	$$
	which implies that $(1-\delta)\beta> \gamma$. Now assume that $\alpha=0$, then $H(0, (x,0))=(x,0) \in \pi_{2}^{-1}((\gamma,1])$. This means that $\gamma<0$. For every $0< \varepsilon<1$, $(0,(x,0)) \in [0,\varepsilon) \times \pi_{2}^{-1}((\gamma,1])$. Let $(\delta, (y,\beta)) \in [0,\varepsilon) \times \pi_{2}^{-1}((\gamma,1])$, then $\delta< \varepsilon$ and $\beta> \gamma$. It follows that $(1-\delta)\beta> (1-\delta)\gamma \geq \gamma$, therefore $H((\delta, (y,\beta)))=(y, (1-\delta)\beta) \in \pi_{2}^{-1}((\gamma,1])$. Finally assume that $(0,(x,\alpha)) \in (T^{\ast})^{-1}((\gamma, 1])$. Since $T(x)-\alpha > \gamma$, then for every $0< \varepsilon<1$, $[0,\varepsilon) \times (T^{\ast})^{-1}((\gamma, 1])$ is an open neighbourhood of $(0, (x,\alpha))$. For every $(\delta, (y,\beta)) \in [0,\varepsilon) \times (T^{\ast})^{-1}((\gamma, 1])$ we see that 
	$$T(y)-(1-\delta)\beta \geq T(y)-\beta> \gamma.$$
	This proves that $H((\delta, (y,\beta)))=(y, (1-\delta)\beta) \in (T^{\ast})^{-1}((\gamma, 1])$, concluding the proof.
\end{proof}

\section{The groupoid approach}

In this section we will consider the topological groupoid $\pi^{l}(X \times J)$ of homotopy classes of paths on $X \times J$ equipped with the Lasso topology. The topological groupoid $\pi^{l}(Y)$ whose elements are homotopy classes of paths on $Y$ has been defined in \cite{LT} under the assumption that $Y$ is path connected (p.c) and locally path connected (l.p.c), therefore the definition of our $\pi^{l}(X \times J)$ would require $X \times J$ to be p.c and l.p.c. It turns out from Proposition \ref{pc+lpc} that in order for $X \times J$ to be p.c and l.p.c., it is enough to assume that $(X, \iota_{X}(\mathcal{T}))$ is p.c and l.p.c. The first part of this section is devoted to the proof of Proposition \ref{pc+lpc}. In order to make the proof easy to follow we have proved in advance a number of technical results one of which is Lemma \ref{s-i} that uses Remark \ref{fhrem}, a result which belongs to the second part of the section. The second part contains also a few other technical results that precede the proof of Theorem \ref{dfrt} which states that $\pi^{l}(X \times J)$ deformation retract onto $\pi^{l}(X \times \{0\})$. The result of Theorem \ref{dfrt} will be essential in section \ref{fc}.

In what follows we will let $\sigma: X \times J \rightarrow X \times \{0\}$ be the restriction of $H$ on $\{1\} \times (X \times J)$. Since $H$ is a deformation retraction, $\sigma$ has to be a homotopy equivalence. Here the topology in $X \times \{0\}$ is the one derived from that in $X \times J$.

\begin{lemma} \label{prsc}
	For every path $p$ in $X \times \{0\}$ and every fixed $\alpha \in J$, we let $\prescript {}{\alpha} p: I \rightarrow X \times J$ defined by $\prescript {}{\alpha}p(u)=(\pi_{1}(p(u)), \alpha)$ where $\pi_{1}: X \times \{0\} \rightarrow X$ is the homeomorphism of Lemma \ref{hmm}. The map $\prescript {}{\alpha} p$ is a path in $X \times J$.
\end{lemma}
\begin{proof}
	If $(T^{\ast})^{-1}(\gamma,1]$ is a sub bases member containing $(\pi_{1}(p(u)), \alpha)$, then $(T \pi_{1} p)(u)> \alpha +\gamma$. But the map $T: X \rightarrow I$ is continuous, therefore $(T \pi_{1} p)^{-1}((\alpha + \gamma, 1]  \cap I)$ is an open subset of $I$ and it contains $u$. For every $v \in (T \pi_{1} p)^{-1}((\alpha + \gamma, 1] \cap I)$, $(T \pi_{1} p)(v)> \alpha + \gamma$, which means that $\prescript {}{\alpha} p(v)=(\pi_{1}(p(v)), \alpha) \in (T^{\ast})^{-1}(\gamma,1]$. If now a sub bases member containing $(\pi_{1}(p(u)), \alpha)$ is $\pi_{2}^{-1}(\gamma,1]$, then $\alpha> \gamma$. In this case as an open neighbourhood of $\alpha$ we can take the whole of $I$ since for every $v \in I$, $\prescript {}{\alpha}p(v)=(\pi_{1}(p(v)), \alpha) \in \pi_{2}^{-1}(\gamma,1]$. It follows that $\prescript {}{\alpha}p$ is continuous.
\end{proof}

\begin{lemma} \label{rho-p}
	For every finite number of open subsets of the form $(T_{i}^{\ast})^{-1}(\gamma_{i},1]$, we have that $\sigma(\underset{i }{\cap} (T_{i}^{\ast})^{-1}(\gamma_{i},1])=\underset{i}{\cap}\sigma((T_{i}^{\ast})^{-1}(\gamma_{i},1])$.
\end{lemma}
\begin{proof}
	We prove the non obvious inclusion $\underset{i}{\cap}\sigma((T_{i}^{\ast})^{-1}(\gamma_{i},1]) \subseteq \sigma(\underset{i}{\cap} (T_{i}^{\ast})^{-1}(\gamma_{i},1])$. Let $(x,0) \in \underset{i}{\cap}\sigma((T_{i}^{\ast})^{-1}(\gamma_{i},1])$, then there are $\alpha_{i} \in J$ such that $T_{i}(x)> \alpha_{i}+\gamma_{i}\geq \gamma_{i}$. This implies that $(x,0) \in \underset{i}{\cap}(T_{i}^{\ast})^{-1}(\gamma_{i},1])$. Since $(x,0)=\sigma(x,0)$, the inclusion follows.
\end{proof}

\begin{lemma} \label{lowerU}
	For every open subset of the form $(T^{\ast})^{-1}(\gamma,1]$ in $X \times J$, every $(y,\beta) \in (T^{\ast})^{-1}(\gamma,1]$ and every $0\leq \beta'\leq \beta$, $(y, \beta') \in (T^{\ast})^{-1}(\gamma,1]$. In particular, if $(y,\beta) \in (T^{\ast})^{-1}(\gamma,1]$, then $(y,0) \in (T^{\ast})^{-1}(\gamma,1]$.
\end{lemma}
\begin{proof}
	Indeed, since $T(y)> \beta + \gamma$, then $T(y)> \beta' + \gamma$, hence $(y, \beta') \in (T^{\ast})^{-1}(\gamma,1]$.
\end{proof}

\begin{lemma} \label{s-i}
	\begin{itemize}
		\item [(i)] If $(y, \beta_{1}), (y, \beta_{2}) \in (T^{\ast})^{-1}(\gamma,1]$, then the path $p$ arising from Remark \ref{fhrem} which connects $(y, \beta_{1})$ with $(y, \beta_{2})$ sits inside $(T^{\ast})^{-1}(\gamma,1]$.
		\item[(ii)] If $(y, \beta_{1}), (y, \beta_{2}) \in \pi_{2}^{-1}(\gamma,1]$, then the path $p$ arising from Remark \ref{fhrem} which connects $(y, \beta_{1})$ with $(y, \beta_{2})$ sits inside $\pi_{2}^{-1}(\gamma,1]$.
	\end{itemize}
\end{lemma}
\begin{proof}
	(i) This follows from Lemma \ref{lowerU} since each point of $p$ is a pair $(y, \beta)$ with $\beta_{1} \leq \beta \leq \beta_{2}$. \\
	(ii) Indeed, any point of $p$ is a pair $(y, \beta)$ with $\beta_{1} \leq \beta \leq \beta_{2}$. But $\beta_{1}> \gamma$, hence $\beta> \gamma$ and $(y, \beta) \in \pi_{2}^{-1}(\gamma,1]$.
\end{proof}

\begin{lemma} \label{rho}
The restriction of $\sigma$ on the sub bases is an open map.
\end{lemma}
\begin{proof}
First we assume that $B=(T^{\ast})^{-1}(\gamma,1]$, with $T \in \mathcal{T}$ and $-1\leq \gamma < 1$. For every $(x,0)=\sigma(x,\alpha)$ with $(x,\alpha) \in (T^{\ast})^{-1}(\gamma,1]$ we have that $T(x)> \alpha+\gamma\geq \gamma$ from which we get that $(x,0) \in (T^{\ast})^{-1}(\gamma,1] \cap (X \times \{0\})$. So we have in this case that 
	$$\sigma((T^{\ast})^{-1}(\gamma,1]) \subseteq (T^{\ast})^{-1}(\gamma,1] \cap (X \times \{0\}).$$
	But the converse inclusion holds true as well for if $(x,0) \in (T^{\ast})^{-1}(\gamma,1] \cap (X \times \{0\})$, then $(x,0)=\sigma(x,0)$. Therefore we have that 
	$$\sigma((T^{\ast})^{-1}(\gamma,1]) = (T^{\ast})^{-1}(\gamma,1] \cap (X \times \{0\}),$$
	and $(T^{\ast})^{-1}(\gamma,1] \cap (X \times \{0\})$ is open in the subspace topology. Secondly, if the sub bases member is  $(\pi^{-1}_{2})(\gamma,1]$, where we can assume without loss of generality that $\gamma \geq 0$, then it follows that $\sigma(\pi^{-1}_{2}(\gamma,1])=X \times \{0\}$ since for $\gamma \geq 0$, $\pi^{-1}_{2}(\gamma,1]=X \times (\gamma,1]$, hence $\sigma(\pi^{-1}_{2}(\gamma,1])$ is open. 
\end{proof}

\begin{proposition} \label{pc+lpc}
	\begin{itemize}
		\item [(i)] $X \times \{0\}$ is l.p.c., if and only if $X \times J$ is l.p.c.
		\item[(ii)] $X \times \{0\}$ is p.c., if and only if $X \times J$ is p.c.
	\end{itemize}
\end{proposition}
\begin{proof}
	(i) Assume that $X \times \{0\}$ is l.p.c. Let $(x, \alpha)$ be arbitrary and $U$ some open subset of $X \times J$ containing $(x, \alpha)$. Then $U$ decomposes as $U=\underset{i \in \mathcal{I}}{\cup}C_{i}$ where each $C_{i}=\underset{j \in \mathcal{J}_{i}}{\cap}B_{i,j}$ with $\mathcal{J}_{i}$ finite and each $B_{i,j}$ being a sub bases member. It follows that for some $i \in \mathcal{I}$, $(x,\alpha) \in C_{i}$. There are two possible cases for $C_{i}$. In the first one, all the intersecting components $B_{i,j}$ of $C_{i}$ are of the form $\pi_{2}^{-1}(\gamma_{i,j},1]$, hence $C_{i}=\pi_{2}^{-1}(\gamma_{i},1]$ where $\gamma_{i}=\text{max}\{\gamma_{i,j}: j \in 
	\mathcal{J}_{i}\}$. Lemma \ref{s-i}, (ii) implies that there is a path connecting any two points $(x_{1}, \alpha_{1}), (x_{2}, \alpha_{2}) \in \pi_{2}^{-1}(\gamma_{i},1]$ which is included in $\pi_{2}^{-1}(\gamma_{i},1]$ and therefore in $U$.
	The second case for $C_{i}$ is when all of the $B_{i,j}$ are of the form $(T_{1}^{\ast})^{-1}(\gamma_{1},1],..., (T_{n}^{\ast})^{-1}(\gamma_{n},1]$. From Lemma \ref{rho-p} and Lemma \ref{rho}, $\sigma((T_{1}^{\ast})^{-1}(\gamma_{1},1]) \cap ... \cap \sigma ((T_{n}^{\ast})^{-1}(\gamma_{n},1])$ is an open neighbourhood of $(x,0)$. From the assumption on $X \times \{0\}$, there is neighbourhood $V^{\ast}$ of $(x,0)$ such that $V^{\ast} \subseteq \sigma((T_{1}^{\ast})^{-1}(\gamma_{1},1]) \cap ... \cap \sigma ((T_{n}^{\ast})^{-1}(\gamma_{n},1])$ with the property that any two points in $V^{\ast}$ can be connected by a path in $\sigma((T_{1}^{\ast})^{-1}(\gamma_{1},1]) \cap ... \cap \sigma ((T_{n}^{\ast})^{-1}(\gamma_{n},1])$. Since $\sigma$ is continuous, $\sigma^{-1}(V^{\ast})$ is an open neighbourhood of $(x,\alpha)$, hence 
$$\sigma^{-1}(V^{\ast}) \cap (T_{1}^{\ast})^{-1}(\gamma_{1},1] \cap ...\cap (T_{n}^{\ast})^{-1}(\gamma_{n},1]$$
is also an open neighbourhood of $(x, \alpha)$ and 
$$\sigma^{-1}(V^{\ast}) \cap (T_{1}^{\ast})^{-1}(\gamma_{1},1] \cap ...\cap (T_{n}^{\ast})^{-1}(\gamma_{n},1] \subseteq (T_{1}^{\ast})^{-1}(\gamma_{1},1] \cap ...\cap (T_{n}^{\ast})^{-1}(\gamma_{n},1].$$
Let now $(x_{1}, \alpha_{1}), (x_{2}, \alpha_{2}) \in \sigma^{-1}(V^{\ast}) \cap (T_{1}^{\ast})^{-1}(\gamma_{1},1] \cap ...\cap (T_{n}^{\ast})^{-1}(\gamma_{n},1]$. It follows for their images under $\sigma$ that
\begin{align*}
	(x_{1}, 0), (x_{2},0) &\in \sigma(\sigma^{-1}(V^{\ast}) \cap (T_{1}^{\ast})^{-1}(\gamma_{1},1] \cap ...\cap (T_{n}^{\ast})^{-1}(\gamma_{n},1]) \\
	&\subseteq V^{\ast} \cap \sigma((T_{1}^{\ast})^{-1}(\gamma_{1},1]) \cap ... \cap \sigma ((T_{n}^{\ast})^{-1}(\gamma_{n},1])\\
	&=V^{\ast}.
\end{align*}
From our assumption, there is a path $p^{\ast}$ connecting $(x_{1},0)$ with $(x_{2},0)$ which is included in $\sigma((T_{1}^{\ast})^{-1}(\gamma_{1},1]) \cap ... \cap \sigma ((T_{n}^{\ast})^{-1}(\gamma_{n},1])$. But
$$\sigma((T_{1}^{\ast})^{-1}(\gamma_{1},1]) \cap ... \cap \sigma ((T_{n}^{\ast})^{-1}(\gamma_{n},1]) \subseteq (T_{1}^{\ast})^{-1}(\gamma_{1},1] \cap ...\cap (T_{n}^{\ast})^{-1}(\gamma_{n},1],$$
hence $p^{\ast}$ is also included in $(T_{1}^{\ast})^{-1}(\gamma_{1},1] \cap ...\cap (T_{n}^{\ast})^{-1}(\gamma_{n},1]$. From Lemma \ref{s-i},(i) there are paths, $q_{1}$ connecting $(x_{1}, \alpha_{1})$ with $(x_{1},0)$, and $q_{2}$ connecting $(x_{2}, \alpha_{2})$ with $(x_{2},0)$, which are inside $(T_{1}^{\ast})^{-1}(\gamma_{1},1] \cap ...\cap (T_{n}^{\ast})^{-1}(\gamma_{n},1]$. It follows that $q_{1} \ast p^{\ast} \ast q^{-1}_{2}$ is a path that connects $(x_{1}, \alpha_{1})$ with $(x_{2},\alpha_{2})$ which sits inside $(T_{1}^{\ast})^{-1}(\gamma_{1},1] \cap ...\cap (T_{n}^{\ast})^{-1}(\gamma_{n},1]$.  

The third case for $C_{i}$ is when some of the $B_{i,j}$ are of the form $(T_{1}^{\ast})^{-1}(\gamma_{1},1],..., (T_{n}^{\ast})^{-1}(\gamma_{n},1]$, and the rest of the components form a non-empty family consisting of open sets $\pi_{2}^{-1}(\delta_{j},1]$ whose intersection has the form $\pi_{2}^{-1}(\delta,1]$, therefore 
$$C_{i}=(T_{1}^{\ast})^{-1}(\gamma_{1},1]\cap \dots \cap (T_{n}^{\ast})^{-1}(\gamma_{n},1] \cap \pi_{2}^{-1}(\delta,1].$$
Since $(x,\alpha) \in C_{i}$ we have that $\alpha > \delta$ and for every $i=1,...,n$, $T_{i}(x)> \gamma_{i}+\alpha$ which means in particular that 
$$(x,0) \in \sigma((T_{1}^{\ast})^{-1}(\gamma_{1}+\alpha,1]) \cap \dots \cap \sigma((T_{n}^{\ast})^{-1}(\gamma_{n}+\alpha,1]).$$
From the assumption on $X \times \{0\}$, there is some path connected open set $V^{\ast}$ such that 
$$(x,0) \in V^{\ast} \subseteq \sigma((T_{1}^{\ast})^{-1}(\gamma_{1}+\alpha,1]) \cap \dots \cap \sigma((T_{n}^{\ast})^{-1}(\gamma_{n}+\alpha,1]).$$
Since $\sigma$ is continuous, then $\sigma^{-1}(V^{\ast})$ is an open neighbourhood of $(x,\alpha)$, then so is 
$$U^{\ast}=\sigma^{-1}(V^{\ast}) \cap (T_{1}^{\ast})^{-1}(\gamma_{1},1]\cap \dots \cap (T_{n}^{\ast})^{-1}(\gamma_{n},1] \cap \pi_{2}^{-1}(\delta,1].$$
We prove that for every two points $(x_{1}, \alpha_{1}), (x_{2}, \alpha_{2}) \in U^{\ast}$, there is a path in $(T_{1}^{\ast})^{-1}(\gamma_{1},1]\cap \dots \cap (T_{n}^{\ast})^{-1}(\gamma_{n},1] \cap \pi_{2}^{-1}(\delta,1]$ connecting these two points. Their images $(x_{1},0), (x_{2},0)$ under $\sigma$ are in $V^{\ast}$, therefore there is a path $p^{\ast}$ connecting $(x_{1},0)$ with $(x_{2},0)$ which is included in $\sigma((T_{1}^{\ast})^{-1}(\gamma_{1}+\alpha,1]) \cap \dots \cap \sigma((T_{n}^{\ast})^{-1}(\gamma_{n}+\alpha,1])$. This means that for every point $(y,0)$ in $p^{\ast}$ we have that for every $i=1,...,n$, $T_{i}(y)> \gamma_{i}+\alpha$ which implies that 
$$(y,\alpha) \in (T_{1}^{\ast})^{-1}(\gamma_{1},1]\cap \dots \cap (T_{n}^{\ast})^{-1}(\gamma_{n},1] \cap \pi_{2}^{-1}(\delta,1].$$
It follows that the path $\prescript{}{\alpha}p^{\ast}$ which connects $(x_{1}, \alpha)$ with $(x_{2}, \alpha)$ is such that
$$Im(\prescript{}{\alpha}p^{\ast}) \subset (T_{1}^{\ast})^{-1}(\gamma_{1},1]\cap \dots \cap (T_{n}^{\ast})^{-1}(\gamma_{n},1] \cap \pi_{2}^{-1}(\delta,1].$$
Since 
$$(x_{1},\alpha), (x_{2}, \alpha), (x_{1}, \alpha_{1}), (x_{2}, \alpha_{2}) \in (T_{1}^{\ast})^{-1}(\gamma_{1},1]\cap \dots \cap (T_{n}^{\ast})^{-1}(\gamma_{n},1] \cap \pi_{2}^{-1}(\delta,1],$$
we have from Lemma \ref{s-i} the existence of paths $p_{1}$ and $p_{2}$ inside $(T_{1}^{\ast})^{-1}(\gamma_{1},1]\cap \dots \cap (T_{n}^{\ast})^{-1}(\gamma_{n},1] \cap \pi_{2}^{-1}(\delta,1]$ which connect respectively $(x_{1}, \alpha_{1})$ with $(x_{1}, \alpha)$, and $(x_{2}, \alpha)$ with $(x_{2}, \alpha_{2})$. The composite $p_{1} \ast \prescript{}{\alpha}p^{\ast} \ast p_{2}$ is a path inside $(T_{1}^{\ast})^{-1}(\gamma_{1},1]\cap \dots \cap (T_{n}^{\ast})^{-1}(\gamma_{n},1] \cap \pi_{2}^{-1}(\delta,1]$ which connects $(x_{1}, \alpha_{1})$ with $(x_{2}, \alpha_{2})$.

Conversely assume that $X \times J$ is l.p.c. Let $U^{\ast}$ be any open in $X \times \{0\}$ containing any $(x,0)$. Consider $\sigma^{-1}(U^{\ast})$ which is open in $X \times J$ and contains $(x,0)$. Since $X \times J$ is l.p.c., there exists an open neighbourhood $V$ of $(x,0)$ such that $V \subseteq \sigma^{-1}(U^{\ast})$ and satisfies the property that any two points in $V$ are connected by a path which is included in $V$. Decompose $V$ as $V=\underset{i \in \mathcal{I}}{\cup}C_{i}$ where each $C_{i}=\underset{j \in \mathcal{J}_{i}}{\cap}B_{i,j}$ with $\mathcal{J}_{i}$ finite and each $B_{i,j}$ being a sub bases member. There is some $i \in \mathcal{I}$ such that $(x,0) \in C_{i}=\underset{j \in \mathcal{J}_{i}}{\cap}B_{i,j}$. If for some $j \in \mathcal{J}_{i}$, $B_{i,j}=\pi_{2}^{-1}(\gamma_{i,j},1]$, then $\gamma_{i,j}<0$, therefore $B_{i,j}=X \times J$. It follows that, if for all $j \in \mathcal{J}_{i}$, $B_{i,j}=\pi_{2}^{-1}(\gamma_{i,j},1]$, then $C_{i}=X \times J$, consequently $V=X\times J=\sigma^{-1}(U^{\ast})$. In this case we prove that $U^{\ast}$ has the property that any two points in $U^{\ast}$ are connected with a path that is included in $U^{\ast}$. Indeed, let $(x_{1},0), (x_{2},0) \in U^{\ast}$. Their preimages $(x_{1},0), (x_{2},0) \in V$ can be connected by a path $p$ in $V=\sigma^{-1}(U^{\ast})$. Then $\sigma(p)$ is a path in $U^{\ast}$ connecting $(x_{1},0)$ with $(x_{2},0)$. Now assume that not for all $j \in \mathcal{J}_{i}$, $B_{i,j}=\pi_{2}^{-1}(\gamma_{i,j},1]$. It follows that there is a non-empty subset $\mathcal{J}'_{i}$ of $\mathcal{J}_{i}$ such that for every $j \in \mathcal{J}'_{i}$, $B_{i,j}=(T^{\ast}_{j})^{-1}(\gamma_{j},1]$ and for all $j \notin \mathcal{J}'_{i}$, $B_{i,j}=\pi_{2}^{-1}(\gamma_{j},1]=X \times J$. In such circumstances, $C_{i}= \underset{j \in \mathcal{J}'_{i}}{\cap}(T^{\ast}_{j})^{-1}(\gamma_{j},1]$. From Lemma \ref{rho-p} and Lemma \ref{rho} we have that $\sigma(C_{i})$ is an open neighbourhood of $(x,0)$ in $X \times \{0\}$. Let $(x_{1},0), (x_{2},0) \in \sigma(C_{i})$ be arbitrary. Since 
	$$\sigma(C_{i})= \sigma(\underset{j \in \mathcal{J}'_{i}}{\cap}(T^{\ast}_{j})^{-1}(\gamma_{j},1])=\underset{j \in \mathcal{J}'_{i}}{\cap}\sigma((T^{\ast}_{j})^{-1}(\gamma_{j},1]))=\underset{j \in \mathcal{J}'_{i}}{\cap}(T^{\ast}_{j})^{-1}(\gamma_{j},1]) \cap (X \times \{0\}),$$ 
	it follows that $(x_{1},0), (x_{2},0) \in \underset{j \in \mathcal{J}'_{i}}{\cap}(T^{\ast}_{j})^{-1}(\gamma_{j},1]) \subseteq  V$. From the assumption on $V$, there is a path $p$ in $V$ connecting $(x_{1},0)$ with $(x_{2},0)$. Since $V \subseteq \sigma^{-1}(U^{\ast})$, it follows that $\sigma(p)$ is a path in $U^{\ast}$ connecting $(x_{1},0)$ with $(x_{2},0)$ proving that $X \times \{0\}$ is locally path connected. \\
	(ii) Assume that $X \times \{0\}$ is path-connected, and let $(x,\alpha), (y, \beta) \in X \times J$ be arbitrary. There is a path $p$ connecting $(x, \alpha)$ with $(x,0)$, and a path $q$ connecting $(y, \beta)$ with $(y,0)$. Also, from the assumption, there is a path $r$ connecting $(x,0)$ with $(y, 0)$. Now the path $p \ast r \ast q^{-1}$ connects $(x,\alpha)$ with $(y, \beta)$ proving that $X \times J$ is path connected. Conversely, for every $(x,0), (y,0) \in X \times \{0\}$, and every path $p$ in $X \times J$ which connects $(x,0)$ with $(y,0)$, $\sigma(p)$ is a path in $X \times \{0\}$ which connects $(x,0)$ with $(y,0)$ proving that $X \times \{0\}$ is path connected. This completes the proof.
\end{proof}

\begin{lemma} \label{ext-p}
	Let $\gamma: I \rightarrow X \times J$ be a path in $X \times J$, and $t \in I$ a fixed element. Let $I \times (X \times J)$ be the product of spaces equipped with the product topology. The map $\gamma_{t}: I \rightarrow I \times (X \times J)$ defined by $\gamma_{t}(s)=(t, \gamma(s))$ is continuous, and therefore a path in $I \times (X \times J)$. This path is in fact a pair $(c_{t},\gamma)$ of paths with $c_{t}$ the constant path at $t$ in $I$, and $\gamma$ the given path in $X \times J$.
\end{lemma}
\begin{proof}
	Since $I \times (X \times J)$ is equipped with the product topology, the continuity of $\gamma_{t}$ follows from the fact that its compositions with projections in the first and in the second coordinate, give respectively, the constant map at $t$, and the map $\gamma$ which are both continuous. The second statement is based on the standard fact that paths in a product of two spaces equipped with the product topology correspond to pairs of paths by taking compositions with the respective projections.
\end{proof}

\begin{lemma} \label{inside-0}
	Let $\gamma: I \rightarrow (X \times J)$ be a path with image inside $X \times \{0\}$ and $t \in I$ arbitrary, then $H\circ \gamma_{t}=\gamma$.
\end{lemma}
\begin{proof}
	Let for any $s \in I$, $(x,0)=\gamma(s)$. Then, 
	$$(H\circ \gamma_{t})(s)=H(t, \gamma(s))=H(t,(x,0))=(x,0)=\gamma(s),$$
	which shows that $H\circ \gamma_{t}=\gamma$.
\end{proof}

\begin{lemma}\label{1-inside}
	Let $\gamma: I \rightarrow (X \times J)$ be any path, then $H\circ \gamma_{1}$ is a path inside $X \times \{0\}$.
\end{lemma}
\begin{proof}
	Let for any $s \in I$, $(x,\alpha)=\gamma(s)$. Then, 
	$$(H\circ \gamma_{1})(s)=H(1, \gamma(s))=H(1,(x,\alpha))=(x,(1-1)\alpha)=(x,0),$$
	which shows that $\text{Im}(H\circ \gamma_{1}) \subseteq X \times \{0\}$.
\end{proof}

\begin{lemma}\label{zero}
	Let $\gamma: I \rightarrow (X \times J)$ be any path, then $H\circ \gamma_{0}=\gamma$.
\end{lemma}
\begin{proof}
	Let for any $s \in I$, $(x,\alpha)=\gamma(s)$. Then, 
	$$(H\circ \gamma_{0})(s)=H(0, \gamma(s))=H(0,(x,\alpha))=(x,(1-0)\alpha)=(x,\alpha)=\gamma(s),$$
	which shows that $H\circ \gamma_{0}=\gamma$.
\end{proof}

\begin{lemma} \label{free homotopy}
For every path $\rho: I \rightarrow X \times J$ and every $s, t \in I$ with $s \neq t$, the paths $H \circ \rho_{s}$ and $H \circ \rho_{t}$ are freely homotopic in $X \times J$. 
\end{lemma}
\begin{proof}
We let first $\kappa(s,t): I \rightarrow I$ be the map $\kappa(s,t)(x)=(t-s)x +s$ which is continuous and satisfies $\kappa(s,t)(0)=s$ and $\kappa(s,t)(1)=t$, hence it is a path in $I$ from $s$ to $t$. Now we define 
	$$\chi_{\rho, \kappa(s,t)}: I \times I \rightarrow X \times J$$
	by setting
	$$\chi_{\rho, \kappa(s,t)}(\eta, x)=H(\kappa(s,t)(x), \rho(\eta)).$$
	For $x=0$ we have that
	$$\chi_{\rho, \kappa(s,t)}(\eta,0)=H(s, \rho(\eta))=(H \circ \rho_{s})(\eta),$$
	and for $x=1$ we have that
	$$\chi_{\rho, \kappa(s,t)}(\eta,1)=H(t, \rho(\eta))=(H \circ \rho_{t})(\eta).$$
	The map $\chi_{\rho, \kappa(s,t)}$ is continuous. Indeed, $\chi_{\rho, \kappa(s,t)}$ factors through $H$ as $\chi_{\rho, \kappa(s,t)}=H \circ (\kappa(s,t), \rho) \circ j$ where $(\kappa(s,t), \rho): I \times I \rightarrow I \times (X \times J)$ is defined by $(x,\eta) \mapsto (\kappa(s,t)(x), \rho(\eta))$ and $j: I \times I \rightarrow I \times I$ is defined by $j(\eta, x)=(x,\eta)$. The continuity of $\chi_{\rho, \kappa(s,t)}$ follows from that of $j$ and that of $(\kappa(s,t), \rho)$. The map $j$ is obviously continuous, and $(\kappa(s,t), \rho)$ is continuous too since its composite with the projection in the first coordinate gives $\kappa(s,t)$ and its composite with the projection in the second coordinate gives $\rho$ which are both continuous.
\end{proof}

\begin{remark} \label{fhrem}
Let $\rho: I \rightarrow X \times J$ be a path with $\iota(\rho)=\rho(0)=(y,\alpha)$ and $\tau(\rho)=\rho(1)=(z,\beta)$. Let $s, t \in I$ and let $\chi_{\rho, \kappa(s,t)}$ be the free homotopy of Lemma \ref{free homotopy} from $H \circ \rho_{s}$ to $H \circ \rho_{t}$. The restriction of the homotopy $\chi_{\rho, \kappa(s,t)}$ on $ \{0\} \times I$ induces a path $p$ given by
$$p(x)=\chi_{\rho, \kappa(s,t)}(0, x)=H(\kappa(s,t)(x),\rho(0)).$$
	We note that
	$$p(0)=(H\circ \rho_{s})(0)=H(s, \rho(0))=(y, (1-s)\alpha)$$ and 
	$$p(1)=(H\circ \rho_{t})(0)=H(t, \rho(0))=(y, (1-t)\alpha).$$ 
	Similarly, by restricting $\chi_{\rho, \kappa(s,t)}$ on $ \{1\} \times I$ we get a path $q$ given by
	$$q(x)=\chi_{\rho, \kappa(s,t)}(1,x)=H(\kappa(s,t)(x),\rho(1)),$$
	which runs from
	$$q(0)=(H\circ \rho_{s})(1)=H(s, \rho(1))=(z, (1-s)\beta)$$ to
	$$q(1)=(H\circ \rho_{t})(1)=H(t, \rho(1))=(z, (1-t)\beta).$$ 
	We observe that $p^{-1} \ast (H \circ \rho_{s}) \ast q$ and $H \circ  \rho_{t}$ are now parallel paths with
	$$\iota(p^{-1} \ast (H \circ \rho_{s}) \ast q)=\iota(p^{-1})= (y, (1-t)\alpha)=(H \circ \rho_{t})(0),$$
	and
	$$\tau(p^{-1} \ast (H \circ \rho_{s}) \ast q)=\tau(q)= (z, (1-t)\beta)=(H \circ \rho_{t})(1).$$
Also we note that the restriction of $\chi_{\rho,\kappa(s,t)}$ on $I \times \{0\}$ gives $H \circ \rho_{s}$ since 
$$\chi_{\rho,\kappa(s,t)}(\eta, 0)=H(\kappa(s,t)(0), \rho(\eta))=H(s, \rho(\eta))=(H \circ \rho_{s})(\eta).$$
In a similar fashion one can verify that the restriction of $\chi_{\rho,\kappa(s,t)}$ on $I \times \{1\}$ gives $H \circ \rho_{t}$.

Finally, we pay a special attention to the case when $\rho$ is the constant path $c_{(y, \alpha)}$ at some $(y, \alpha)$, $s=0$ and $t \in I$ arbitrary. In this case the above path $p$ starts at $p(0)=H(0, (y, \alpha))=(y, \alpha)$ and ends at $p(1)=H(t, (y, \alpha))=(y, (1-t)\alpha)$. This implies that for every two points $(y, \alpha), (y, \beta) \in X \times J$, there is always a path in $X \times J$ connecting $(y, \alpha)$ with $(y, \beta)$ consisting of points $(y, \delta)$ with $\delta$ between $\alpha$ and $\beta$. Indeed, assuming that one of $\alpha$ or $\beta$ is non zero, for instance $\alpha \neq 0$, we can take $s=0$ and $t=1-\dfrac{\beta}{\alpha}$ and that will do.
\end{remark}

\begin{lemma} \label{relative homotopy}
	The paths $p^{-1} \ast (H \circ \rho_{s}) \ast q$ and $H \circ  \rho_{t}$ are homotopic relative to their end points. 
\end{lemma} 
\begin{proof}
The map $\chi_{\rho, \kappa(s,t)}$ is null homotopic since $I \times I$ is contractible. On the other hand, $I \times I$ is homeomorphic with $D^{2}$ and its boundary is homeomorphic with $S^{1}$. Since the inclusion map $\iota:S^{1} \rightarrow D^{2}$ is null homotopic, it follows that the restriction of $\chi_{\rho, \kappa(s,t)}$ on the boundary of $I \times I$ is null homotopic too. But from Remark \ref{fhrem} the image of this restriction is the path $p^{-1} \ast (H \circ \rho_{s} )\ast q \ast (H \circ \rho_{t})^{-1}$, hence $p^{-1} \ast (H \circ \rho_{s} )\ast q \ast (H \circ \rho_{t})^{-1} \sim c_{(y, (1-t)\alpha)}$. This implies that 
\begin{align*}
p^{-1} \ast (H \circ \rho_{s} )\ast q & \sim p^{-1} \ast (H \circ \rho_{s} )\ast q \ast c_{(z,(1-t)\beta)}\\
&\sim p^{-1} \ast (H \circ \rho_{s} )\ast q \ast (H \circ \rho_{t})^{-1}\ast (H \circ \rho_{t}) \\
&\sim c_{(y, (1-t)\alpha)} \ast (H \circ \rho_{t})\\
& \sim H \circ \rho_{t},
\end{align*}
which proves the claim.
\end{proof}

\begin{lemma} \label{hginv}
	For every path $\gamma: I \rightarrow (X \times J)$ and every $t \in I$, $(H \circ \gamma_{t})^{-1}=H \circ \gamma^{-1}_{t}$.
\end{lemma}
\begin{proof}
	We wee that for every $u \in I$,
	\begin{align*}
		(H \circ \gamma_{t})^{-1}(u)&= (H \circ \gamma_{t})(u-1)\\
		&=H(t, \gamma(u-1))\\
		&=H(t, \gamma^{-1}(u))\\
		&=(H \circ \gamma^{-1}_{t})(u),
	\end{align*}
	proving the claim.
\end{proof}

\begin{lemma} \label{ast com comp}
	For every $n \geq 2$ paths $\gamma_{1},...,\gamma_{n}$ in $X \times J$ such that $\gamma_{1} \ast \dots \ast \gamma_{n}$ exists, and every $t \in I$, $(H \circ (\gamma_{1})_{t}) \ast \dots \ast (H \circ (\gamma_{n})_{t})$ exists, and $H \circ (\gamma_{1} \ast \dots \ast \gamma_{n})_{t}=(H \circ (\gamma_{1})_{t}) \ast \dots \ast (H \circ (\gamma_{n})_{t})$.
\end{lemma}
\begin{proof}
	Indeed, since for every $1\leq i \leq n-1$, $\gamma_{i}(1)=\gamma_{i+1}(0)$, then 
	$$(H \circ (\gamma_{i})_{t})(1)=H(t, \gamma_{i}(1))=H(t, \gamma_{i+1}(0))=(H \circ (\gamma_{i+1})_{t})(0).$$
	The proof of the second part will be done dy induction on $n \geq 2$. When $n=2$, for every $s \in [0, 1/2]$,
	\begin{align*}
		(H \circ (\gamma_{1} \ast \gamma_{2})_{t})(s)&=H(t, (\gamma_{1} \ast \gamma_{2})(s))=H(t, \gamma_{1}(2s))\\
		&=((H \circ (\gamma_{1})_{t}) \ast (H \circ (\gamma_{2})_{t}))(s),
	\end{align*} 
	and similarly, for $s \in [1/2, 1]$,
	\begin{align*}
		(H \circ (\gamma_{1} \ast \gamma_{2})_{t})(s)&=H(t, (\gamma_{1} \ast \gamma_{2})(s))=H(t, \gamma_{2}(2s-1))\\
		&=((H \circ (\gamma_{1})_{t}) \ast (H \circ (\gamma_{2})_{t}))(s).
	\end{align*} 
	By comparing we see that $H \circ (\gamma_{1} \ast \gamma_{2})_{t}=(H \circ (\gamma_{1})_{t}) \ast (H \circ (\gamma_{2})_{t})$. For the inductive step,
	\begin{align*}
		H \circ (\gamma_{1} \ast \dots \ast \gamma_{n-1} \ast \gamma_{n})_{t}&= (H \circ (\gamma_{1} \ast \dots \ast \gamma_{n-1})_{t} )\ast (H \circ (\gamma_{n})_{t}) && (\text{by the base step})\\
		&= (H \circ (\gamma_{1})_{t}) \ast \dots \ast (H \circ (\gamma_{n-1})_{t}) \ast (H \circ (\gamma_{n})_{t}) && (\text{by the assumption.})
	\end{align*}
\end{proof}\\
\\
From this moment and on we will assume that $(X, \iota_{X}(\mathcal{T}))$ is p.c and l.p.c. to ensure that we can define the topological space $\pi^{l}(X \times J)$.

\begin{theorem} \label{dfrt}
The space $\pi^{l}(X \times J)$ deformation retracts to its subspace $\pi^{l}(X \times \{0\})$.
\end{theorem}
\begin{proof}
We define
$$H^{l}: I \times \pi^{l}(X \times J) \rightarrow \pi^{l}(X \times J) \text{ by } (t, [\gamma]) \mapsto [H \circ \gamma_{t}].$$
This map is well defined for if $\gamma \sim \delta$ are homotopic paths in $X \times J$ (relative to end points) and $t\in I$ fixed, then 
$$\pi_{2} \circ \gamma_{t}=\gamma \sim \delta = \pi_{2} \circ \delta_{t}.$$
But for both, $\gamma$ and $\delta$, $$\pi_{1}(\gamma_{t})=c_{t}=\pi_{1}(\delta_{t}),$$
where $c_{t}$ is the constant path at $t$ in $I$, therefore we have a homotopy of pairs
$$\gamma_{t}=(c_{t}, \gamma) \sim (c_{t}, \delta)=\delta_{t}$$
in $I \times (X \times J)$. The continuity of $H$ now implies that  $H \circ \gamma_{t} \sim H \circ \delta_{t}$,
and so $[H \circ \gamma_{t}]=[H \circ \delta_{t}]$.

Before we prove the continuity of $H^{l}$ we define an open cover $\mathcal{U}'$ of $X \times J$ in terms of any given open cover $\mathcal{U}$ of $X \times J$ and of $H$ in the following way. Since $H$ is continuous, for every $V \in \mathcal{U}$, $H^{-1}(V)$ is open in $I \times (X \times J)$, therefore for some index set $\mathfrak{I}_{V}$, $H^{-1}(V)=\underset{i \in \mathfrak{I}_{V}}{\cup}V_{i}^{(1)} \times V_{i}^{(2)}$ where $V_{i}^{(1)}$ is open in $I$ and $V_{i}^{(2)}$ is open in $X \times J$. The family $\mathcal{U}'$ consisting of all $V_{i}^{(2)}$ with $V$ ranging in $\mathcal{U}$ and $i \in \mathfrak{I}_{V}$ is an open cover for $X \times J$. Indeed, for every $(x,\alpha) \in X \times J$, there is $V \in \mathcal{U}$ such that $(x,\alpha) \in V$. It follows that $(0,(x,\alpha)) \in H^{-1}(V)=\underset{i \in \mathfrak{I}_{V}}{\cup}V_{i}^{(1)} \times V_{i}^{(2)}$ therefore for some $i \in \mathfrak{I}_{V}$, $(0,(x,\alpha)) \in V_{i}^{(1)} \times V_{i}^{(2)}$, hence $(x,\alpha) \in V_{i}^{(2)}$, where from the definition, $V_{i}^{(2)} \in \mathcal{U}'$.

Now we prove the continuity of $H^{l}$. Let $(t,[\gamma]) \in I \times \pi^{l}(X \times J)$ and assume that $N([H \circ \gamma_{t}], \mathcal{U}, V,W)$ is an arbitrary open neighbourhood of $[H \circ \gamma_{t}]$. If $\iota(\gamma)=(x,\alpha)$ and $\tau(\gamma)=(y,\beta)$, then $\iota(H \circ \gamma_{t})=(x,(1-t)\alpha)$ and $\tau(H \circ \gamma_{t})=(y,(1-t)\beta)$. So $V$ is an open neighbourhood of $(x,(1-t)\alpha)$ and $W$ an open neighbourhood of $(y,(1-t)\beta)$. It follows that $(t, (x,\alpha)) \in H^{-1}(V)=\underset{i \in \mathfrak{I}_{V}}{\cup}V_{i}^{(1)} \times V_{i}^{(2)}$, therefore for some $i \in \mathfrak{I}_{V}$, $t \in V_{i}^{(1)}$ and $(x,\alpha) \in V_{i}^{(2)}$. Since $V_{i}^{(2)}$ is an open neighbourhood of $(x,\alpha)$, we rewrite it for simplicity as $V_{(x,\alpha)}$. It is also clear from the definition of $\mathcal{U}'$ that $V_{(x,\alpha)} \in \mathcal{U}'$. In the same way we find an open neighbourhood $W_{j}^{(1)}$ of $t$ in $I$ and an open neighbourhood $W_{(y,\beta)}$ of $(y, \beta)$ in $X \times J$ which is a member of $\mathcal{U}'$. Since $V_{i}^{(1)} \cap W_{j}^{(1)}$ is an open subset of $I$ containing $t$, there is an open "interval" $\mathfrak{I}_{t}$ such that $t \in \mathfrak{I}_{t}$ and $\mathfrak{I}_{t} \subset V_{i}^{(1)} \cap W_{j}^{(1)}$. It is now clear that an open neighbourhood of $(t, [\gamma])$ in $I \times \pi^{l}(X \times J)$ is the product $\mathfrak{I}_{t} \times N([\gamma], \mathcal{U}', V_{(x,\alpha)}, W_{(y,\beta)})$. Next we prove that for every $(s,[\delta]) \in \mathfrak{I}_{t} \times N([\gamma], \mathcal{U}', V_{(x,\alpha)}, W_{(y,\beta)})$ we have that $H^{l}(s,[\delta])=[H \circ \delta_{s}] \in N([H \circ \gamma_{t}], \mathcal{U}, V,W)$. From the definition of $N([\gamma], \mathcal{U}', V_{(x,\alpha)}, W_{(y,\beta)})$, 
\begin{equation} \label{delta sim}
\delta \sim \lambda \ast \mu \ast \gamma \ast \mu' \ast \rho
\end{equation}
where $\lambda: I \rightarrow V_{(x,\alpha)}$ with $\tau(\lambda)=(x,\alpha)$, $\mu \in \pi_{1}(\mathcal{U}', (x,\alpha))$, $\mu' \in \pi_{1}(\mathcal{U}', (y,\beta))$ and $\rho: I \rightarrow W_{(y, \beta)}$ with $\iota(\rho)=(y, \beta)$. It follows from the continuity of $H$ and from Lemma \ref{ast com comp} that
\begin{equation} \label{2bcvh}
H \circ \delta_{s} \sim (H \circ \lambda_{s}) \ast (H \circ \mu_{s}) \ast (H \circ \gamma_{s}) \ast (H \circ \mu_{s}') \ast (H \circ \rho_{s}).
\end{equation}
Since $s \in \mathfrak{I}_{t}$, and since for every $u \in I$, $\lambda(u) \in V_{(x,\alpha)}$, it follows that
\begin{equation} \label{2-c}
(H \circ \lambda_{s})(u)=H(s, \lambda(u)) \subseteq H(V_{i}^{(1)} \times V_{(x,\alpha)}) \subseteq H(H^{-1}(V))=V,
\end{equation}
hence $H \circ \lambda_{s}$ is a path with $Im(H \circ \lambda_{s}) \subseteq V$. Also we see that $\tau(H \circ \lambda_{s})=(x,(1-s)\alpha)$ since $\lambda(1)=(x,\alpha)$. By a similar argument we see that $H \circ \rho_{s}$ is a path with $\text{Im}(H \circ \rho_{s}) \subseteq W$ and $\iota(H \circ \rho_{s})=(y, (1-s)\beta)$. Further, since $\mu$ is a finite concatenation of paths $$\mu=\underset{i}{\prod}u_{i} \ast v_{i} \ast u^{-1}_{i}$$ 
where each $u_{i}$ is a path starting at $(x,\alpha)$ and ending at some $(x_{i}, \alpha_{i})$ and $v_{i}$ is a loop at $(x_{i}, \alpha_{i})$ which lies entirely inside some $U'_{i} \in \mathcal{U}'$, it follows that
\begin{align*}
H \circ \mu_{s}&=\underset{i}{\prod}(H \circ (u_{i})_{s}) \ast (H \circ (v_{i})_{s}) \ast (H \circ (u^{-1}_{i})_{s})\stepcounter{equation}\tag{\theequation}\label{Hcirc}\\
&=\underset{i}{\prod}(H \circ (u_{i})_{s}) \ast (H \circ (v_{i})_{s}) \ast (H \circ (u_{i})_{s})^{-1} && (\text{by lemma \ref{hginv}})
\end{align*}
where 
$$(H \circ (u_{i})_{s})(0)=H(s, (x,\alpha))=(x, (1-s)\alpha),$$
and that
$$(H \circ (u_{i})_{s})(1)=H(s, (x_{i},\alpha_{i}))=(x_{i}, (1-s)\alpha_{i}).$$
Since $U'_{i} \in \mathcal{U}'$, there has been some $U_{i} \in \mathcal{U}$ and an open subset $A \subseteq I$ such that $A \times U'_{i}$ takes part in the decomposition of $H^{-1}(U_{i})$ as a union of direct products. If $s \in A$, then for every $r \in I$, since $v_{i}(r) \in U'_{i}$, it follows that $(s,v_{i}(r)) \in H^{-1}(U_{i})$, consequently we have that
$$(H \circ (v_{i})_{s})(r)=H(s, v_{i}(r)) \in H(H^{-1}(U_{i})) \subseteq U_{i}.$$
So in this case the loop $H \circ (v_{i})_{s}$ is included in $U_{i} \in \mathcal{U}$. Now we deal with the case when $s \notin A$. For any $s_{i} \in A$, we have from the above that $H \circ (v_{i})_{s_{i}}$ is included in $U_{i}$. Let $\kappa(s_{i},s): I \rightarrow I$ be the path in $I$ of Lemma \ref{free homotopy} with $\kappa(s_{i},s)(0)=s_{i}$ and $\kappa(s_{i},s)(1)=s$ which defines a free homotopy $\chi_{v_{i}, \kappa(s_{i},s)}$ from $H \circ (v_{i})_{s_{i}}$ to $H \circ (v_{i})_{s}$ by the rule
$$\chi_{v_{i}, \kappa(s_{i},s)}(r,p)=H(\kappa(s_{i},s)(p), v_{i}(r)).$$
For $p=0$ we have that
$$\chi_{v_{i}, \kappa(s_{i},s)}(r,0)=H(s_{i}, v_{i}(r))=(H \circ (v_{i})_{s_{i}})(r),$$
and for $p=1$ we have that
$$\chi_{v_{i}, \kappa(s_{i},s)}(r,1)=H(s, v_{i}(r))=(H \circ (v_{i})_{s})(r).$$
We observe that $H \circ (v_{i})_{s_{i}}$ is a loop at 
$$
(H \circ (v_{i})_{s_{i}})(0)=H(s_{i}, v_{i}(0))=H(s_{i},(x_{i}, \alpha_{i}))=(x_{i}, (1-s_{i})\alpha_{i}),
$$
and most importantly that, 
$$(H \circ (v_{i})_{s_{i}})(r)=H(s_{i}, v_{i}(r)) \in H(H^{-1}(U_{i})) \subseteq U_{i},$$
since $s_{i} \in A$.
The restriction of $\chi_{v_{i}, \kappa(s_{i},s)}$ on $\{0\} \times I$ gives a path $p_{i}$
$$\xymatrix{(x_{i}, (1-s_{i})\alpha_{i})=H(s_{i}, v_{i}(0))=\chi_{v_{i}, \kappa(s_{i},s)}(0,0) \ar[r] & \chi_{v_{i}, \kappa(s_{i},s)}(0,1)}=H(s,v_{i}(0))=(x_{i}, (1-s)\alpha_{i}).$$
But from Lemma \ref{relative homotopy} we have that
$$H \circ (v_{i})_{s} \sim p^{-1}_{i} \ast (H \circ (v_{i})_{s_{i}}) \ast p_{i}.$$
It follows that we can now replace in (\ref{Hcirc}) each component 
$$(H \circ (u_{i})_{s}) \ast (H \circ (v_{i})_{s}) \ast (H \circ (u_{i})_{s})^{-1}$$ 
whenever $s \notin A$ by
$$((H \circ (u_{i})_{s}) \ast p^{-1}_{i}) \ast ( H \circ (v_{i})_{s_{i}}) \ast ((H \circ (u_{i})_{s})\ast p^{-1}_{i})^{-1}$$
without changing the homotopy class of $H \circ \mu_{s}$. In conclusion, $H \circ \mu_{s}$ is homotopic with a loop in $\pi_{1}(\mathcal{U}, (x,(1-s)\alpha))$. By a similar argument, we see that $H \circ \mu'_{s}$ is homotopic with a loop in $\pi_{1}(\mathcal{U}, (y,(1-s)\beta))$. What we proved so far, together with (\ref{2bcvh}), imply that $ [H \circ \delta_{s}]\in N([H \circ \gamma_{s}], \mathcal{U}, V, W)$. A particular case of this is when $\delta=\gamma$, and $\lambda, \mu, \mu', \rho$ are trivial. Then from (\ref{2-c}) we would have for every $\xi \in \mathfrak{I}_{t}$ that
\begin{equation} \label{entirely}
(x, (1-\xi)\alpha)=H(\xi, (x,\alpha))=H(\xi, c_{(x,\alpha)}(u))=(H \circ (c_{(x,\alpha)})_{\xi})(u) \subseteq V.
\end{equation}
We prove now that $[H \circ \gamma_{t}] \in N([H \circ \gamma_{s}], \mathcal{U}, V, W)$ which together with Lemma 2.3 of \cite{LT} imply that $[H \circ \delta_{s}] \in N([H \circ \gamma_{t}], \mathcal{U}, V, W)$ as desired. We will utilize again Remark \ref{fhrem} in the following form. Let $\kappa(s,t): I \rightarrow I$ be the map of Lemma \ref{free homotopy} such that $\kappa(s,t)(0)=s$ and $\kappa(s,t)(1)=t$ and let $\chi_{\gamma, \kappa(s,t)}$ be the corresponding free homotopy from $H \circ \gamma_{s}$ to $H \circ \gamma_{t}$. The restriction of $\chi_{\gamma, \kappa(s,t)}$ on $\{0\} \times I$ gives a path 
$$p(\xi)=\chi_{\gamma, \kappa(s,t)}(0,\xi)=H(\kappa(s,t)(\xi), \gamma(0))$$
with 
$$p(0)=H(s, \gamma(0))=(x, (1-s)\alpha) \text{ and } p(1)=H(t, \gamma(0))=(x, (1-t)\alpha),$$
which from (\ref{entirely}) is included entirely in $V$. By restricting now $\chi_{\gamma, \kappa(s,t)}$ on $\{1\} \times I$ we get a path 
$$q(\xi)=\chi_{\gamma, \kappa(s,t)}(1,\xi)=H(\kappa(s,t)(\xi), \gamma(1)),$$
with
$$q(0)=H(s, \gamma(1))=(y, (1-s)\beta) \text{ and } q(1)=H(t, \gamma(1))=(y, (1-t)\beta).$$
By a similar argument with the one used to prove that $\text{Im}p \subset U$, one can prove that $\text{Im}q \subset W$. It follows that $[p^{-1} \ast (H \circ \gamma_{s}) \circ q] \in N([H \circ \gamma_{s}], \mathcal{U}, V, W)$. Now Lemma \ref{relative homotopy} implies that 
$$H \circ \gamma_{t} \sim p^{-1} \ast (H \circ \gamma_{s}) \circ q,$$
which proves that $[H \circ \gamma_{t}] \in N([H \circ \gamma_{s}], \mathcal{U}, V, W)$. Finally, we need to check the rest of the conditions for $H^{l}$ to be a deformation retraction. For every $[\gamma] \in H^{l}(X \times J)$ we have from Lemma \ref{zero} that
$$H^{l}(0, [\gamma])=[H \circ \gamma_{0}]=[\gamma].$$
For every $[\gamma] \in H^{l}(X \times J)$ we have from Lemma \ref{1-inside} that
$$H^{l}(1, [\gamma]) \in H^{l}(X \times \{0\}).$$
For every $t \in I$ and every path $\gamma$ in $X \times \{0\}$ we have from Lemma \ref{inside-0} that
$$H^{l}(t, [\gamma])=[\gamma],$$
which concludes the proof.
\end{proof}

\begin{corollary}
	There is a continuous map $\pi^{l}\sigma: \pi^{l}(X \times J, \iota(\mathcal{T})) \rightarrow \pi^{l}(X \times \{0\}, \iota(\mathcal{T}))$.
\end{corollary}
\begin{proof}
	The existence of $\pi^{l}\sigma$ follows directly from Proposition 2.6 of \cite{LT} and is defined by $\pi^{l}\sigma([\alpha])=[
	\sigma(\alpha)]$ for every homotopy class of a path $\alpha$ in $X \times J$. Also $\pi^{l}\sigma$ can be realized as the restriction of $H^{l}$ on $\{1\} \times \pi^{l}(X \times J)$ since $H^{l}(1, [\alpha])= [H \circ \alpha_{1}]$ and 
	$(H\circ \alpha_{1})(s)=H(1,\alpha(s))=\sigma(\alpha)(s)$.
\end{proof}

\section{Fuzzy complements in topological terms} \label{fc}

As we mentioned in section \ref{rt} the main goal of this paper is interpreting in topological terms the complement of a fuzzy subset. It turns out that the groupoid $\pi(\pi^{l}(X \times J))$ has the necessary information encoded in it to make this interpretation possible. Theorem \ref{main} explained in simple terms, states that every fuzzy subset $F$ and its complement $1-F$, give rise to a pair of functors $[F], [1-F]$ in the category $\textbf{Grpd}$ of groupoids which are related by the equality $[1-F]=\pi(\boldsymbol{\iota}) \circ [F]$ where $\pi(\boldsymbol{\iota}): \pi(\pi^{l}(X \times J)) \rightarrow \pi({\pi^{l}(X \times J)}$ is the natural extension to the homotopy classes of paths of the map $\boldsymbol{\iota}: \pi^{l}(X \times J) \rightarrow \pi^{l}(X \times J)$ which maps every $[\alpha]$ to $[\alpha^{-1}]$. Therefore complementing $F$ in $[0,1]^{X}$, is represented in $\textbf{Grpd}$ by applying $\pi(\boldsymbol{\iota})$ to $[F]$ where $\pi(\boldsymbol{\iota})$ maps points to their inverses. We emphasize here that we do not assume any topology defined on $\pi(\pi^{l}(X \times J))$. It is just the fundamental groupoid of the topological space $\pi^{l}(X \times J)$.

Let $\gamma$ be a path in $X \times J$, $s,t \in I$ and let $\chi_{\gamma,\kappa(s,t)}$ be the free homotopy of Lemma \ref{free homotopy} which transforms $H \circ \gamma_{s}$ to $H \circ \gamma_{t}$. For every $0\leq a< b \leq 1$ we let $\gamma':I \rightarrow X \times J$ be the path given by $\gamma'(\eta)= \gamma(a + \eta(b-a))$. In the following lemma we will describe the free homotopy of Lemma \ref{free homotopy} which transforms $H \circ \gamma'_{s}$ to $H \circ \gamma'_{t}$ in terms of $\chi_{\gamma,\kappa(s,t)}$. Our interest lies in studying the neighbourhoods of the homotopy classes of paths of Remark \ref{hmm} arising from the above homotopy. For this we need the following. For every fixed $\eta \in I$ we let $\ell_{\eta}:I \rightarrow I \times I$ be the continuous map $x \mapsto (\eta,x)$. Its composite $\chi_{\gamma,\kappa(s,t)} \circ \ell_{\eta}$ is continuous and is given by the rule $(\chi_{\gamma,\kappa(s,t)} \circ \ell_{\eta})(x)=H(\kappa(s,t)(x), \gamma (\eta))=(H \circ \gamma_{\kappa(s,t)(x)})(\eta)$. 

\begin{lemma} \label{path-res}
	Let $\gamma$ be a path in $X \times J$, $s,t \in I$ and let $\chi_{\gamma,\kappa(s,t)}$ be the free homotopy of Lemma \ref{free homotopy} which transforms $H \circ \gamma_{s}$ to $H \circ \gamma_{t}$. For every $0 \leq a< b \leq 1$, the free homotopy induced by $\chi_{\gamma,\kappa(s,t)}$ from the restriction of $H \circ \gamma_{s}$ in $[a,b]$ to the restriction of $H \circ \gamma_{t}$ in $[a,b]$ coincides with the free homotopy from $H \circ \gamma'_{s}$ to $H \circ \gamma'_{t}$. The path of Remark \ref{fhrem} arising from this homotopy which connects $(H \circ \gamma_{s})(a)$ with $(H \circ \gamma_{t})(a)$ is given by $(\chi_{\gamma,\kappa(s,t)} \circ \ell_{a})(x)$, and the other path which connects $(H \circ \gamma_{s})(b)$ with $(H \circ \gamma_{t})(b)$ is given by $(\chi_{\gamma,\kappa(s,t)} \circ \ell_{b})(x)$.
\end{lemma}
\begin{proof}
	The first part of the proof is standard and works in general but we have described it in full for the sake of the second statement of the lemma which will be used in the subsequent lemma. The map $\sigma: I \times I \rightarrow I \times I$ defined by $\sigma(\eta, x)= (a + \eta(b-a),x)$ is continuous and its composite with $\chi_{\gamma,\kappa(s,t)}$ is given by 
	$$(\chi_{\gamma,\kappa(s,t)} \circ \sigma)(\eta, x)=H(\kappa(s,t)(x), \gamma(a+\eta(b-a))).$$
	At $x=0$ the corresponding path is $(\chi_{\gamma,\kappa(s,t)} \circ \sigma)(\eta, 0)=H(s, \gamma(a+\eta(b-a)))$ whose initial is 
	$$(\chi_{\gamma,\kappa(s,t)} \circ \sigma)(0, 0)=H(s, \gamma(a))=(H \circ \gamma_{s})(a),$$
	and its terminal is
	$$(\chi_{\gamma,\kappa(s,t)} \circ \sigma)(1, 0)=H(s, \gamma(b))=(H \circ \gamma_{s})(b).$$
	At $x=1$ the corresponding path is $(\chi_{\gamma,\kappa(s,t)} \circ \sigma)(\eta, 1)=H(t, \gamma(a+\eta(b-a)))$ whose initial is 
	$$(\chi_{\gamma,\kappa(s,t)} \circ \sigma)(0, 1)=H(t, \gamma(a))=(H \circ \gamma_{t})(a),$$
	and its terminal is
	$$(\chi_{\gamma,\kappa(s,t)} \circ \sigma)(1, 1)=H(t, \gamma(b))=(H \circ \gamma_{t})(b).$$
	Further we observe that 
	$$H(\kappa(s,t)(x), \gamma(a+\eta(b-a)))=H(\kappa(s,t)(x), \gamma'(\eta)),$$
	where the right hand side gives the homotopy from $H \circ \gamma'_{s}$ to $H \circ \gamma'_{t}$. Now we prove the second part of the lemma. From Remark \ref{fhrem} we see that the path which connects $(H \circ \gamma_{s})(a)=(H \circ \gamma'_{s})(0)$ with $(H \circ \gamma_{t})(a)=(H \circ \gamma'_{t})(0)$ is given by 
	$$p(x)=H(\kappa(s,t)(x), \gamma'(0))=H(\kappa(s,t)(x), \gamma(a))=(\chi_{\gamma,\kappa(s,t)} \circ \ell_{a})(x),$$
	and similarly, the path which connects $(H \circ \gamma_{s})(b)=(H \circ \gamma'_{s})(1)$ with $(H \circ \gamma_{t})(b)=(H \circ \gamma'_{t})(1)$ is given by 
	$$q(x)=H(\kappa(s,t)(x), \gamma'(1))=H(\kappa(s,t)(x), \gamma(b))=(\chi_{\gamma,\kappa(s,t)} \circ \ell_{b})(x),$$
	which concludes the proof.
\end{proof}

\begin{lemma} \label{t-chi}
	The map $\tilde{\chi}_{\gamma,\kappa(s,t)}: I \rightarrow \pi^{l}(X \times I)$ given by $\tilde{\chi}_{\gamma,\kappa(s,t)}(\eta)=[\chi_{\gamma,\kappa(s,t)} \circ \ell_{\eta}]$ is continuous.
\end{lemma}
\begin{proof}
	Let $N([\chi_{\gamma,\kappa(s,t)} \circ \ell_{\eta}], \mathcal{U}, V, W)$ be any open neighbourhood of $[\chi_{\gamma,\kappa(s,t)} \circ \ell_{\eta}]$. Here $V$ is an open set belonging to $\mathcal{U}$ which contains $(\chi_{\gamma,\kappa(s,t)} \circ \ell_{\eta})(0)=(H \circ \gamma_{\kappa(s,t)(0)})(\eta)$, and similarly, $W$ is an open set belonging to $\mathcal{U}$ which contains $(\chi_{\gamma,\kappa(s,t)} \circ \ell_{\eta})(1)=(H \circ \gamma_{\kappa(s,t)(1)})(\eta)$. It follows from the continuity of $H \circ \gamma_{\kappa(s,t)(0)}$ that $(H \circ \gamma_{\kappa(s,t)(0)})^{-1}(V)$ is an open neighbourhood of $\eta$ in $I$, and similarly from the continuity of $H \circ \gamma_{\kappa(s,t)(1)}$ we have that $(H \circ \gamma_{\kappa(s,t)(1)})^{-1}(W)$ is another open neighbourhood of $\eta$ in $I$. Hence $\mathcal{O}=(H \circ \gamma_{\kappa(s,t)(0)})^{-1}(V) \cap (H \circ \gamma_{\kappa(s,t)(1)})^{-1}(W)$ is an open neighbourhood of $\eta$ in $I$. There is an open interval $\mathcal{I}_{\eta} \subseteq \mathcal{O}$ which contains $\eta$. We prove now that $\tilde{\chi}_{\gamma,\kappa(s,t)}(\mathcal{\mathcal{I}_{\eta}}) \subseteq N([\chi_{\gamma,\kappa(s,t)} \circ \ell_{\eta}], \mathcal{U}, V, W)$ from which the continuity of $\tilde{\chi}_{\gamma,\kappa(s,t)}$ follows. Let $\eta' \in \mathcal{\mathcal{I}_{\eta}}$ be arbitrary and assume that $\eta< \eta'$. We observe that for every $\eta \leq \eta" \leq \eta'$
	$$(H \circ \gamma_{\kappa(s,t)(0)})(\eta'') \in (H \circ \gamma_{\kappa(s,t)(0)})(\mathcal{\mathcal{I}_{\eta}})\subseteq (H \circ \gamma_{\kappa(s,t)(0)})((H \circ \gamma_{\kappa(s,t)(0)})^{-1}(V))=V,$$
	which proves that the image of the restriction of $H \circ \gamma_{\kappa(s,t)(0)}$ in $[\eta, \eta']$ lies inside $V$. Similarly, we see that
	$$(H \circ \gamma_{\kappa(s,t)(1)})(\eta'') \in (H \circ \gamma_{\kappa(s,t)(1)})(\mathcal{\mathcal{I}_{\eta}})\subseteq (H \circ \gamma_{\kappa(s,t)(1)})((H \circ \gamma_{\kappa(s,t)(1)})^{-1}(W))=W,$$
	which implies that the image of the restriction of $H \circ \gamma_{\kappa(s,t)(1)}$ in $[\eta, \eta']$ lies inside $W$. From Lemma \ref{free homotopy} the restrictions of the paths $H \circ \gamma_{\kappa(s,t)(0)}$ and $H \circ \gamma_{\kappa(s,t)(1)}$ in $[\eta, \eta']$ are homotopic, and from Lemma \ref{path-res}, the homotopy class of the path connecting $(H \circ \gamma_{\kappa(s,t)(0)})(\eta)$ with $(H \circ \gamma_{\kappa(s,t)(1)})(\eta)$ is $[\chi_{\gamma,\kappa(s,t)} \circ \ell_{\eta}]=\tilde{\chi}_{\gamma,\kappa(s,t)}(\eta)$. Similarly, the homotopy class of the path connecting $(H \circ \gamma_{\kappa(s,t)(0)})(\eta')$ with $(H \circ \gamma_{\kappa(s,t)(1)})(\eta')$ is $[\chi_{\gamma,\kappa(s,t)} \circ \ell_{\eta'}]=\tilde{\chi}_{\gamma,\kappa(s,t)}(\eta')$. Lemma \ref{relative homotopy} implies that 
	$$(\chi_{\gamma,\kappa(s,t)} \circ \ell_{\eta})^{-1} \ast (H \circ \gamma'_{\kappa(s,t)(0)}) \ast (\chi_{\gamma,\kappa(s,t)} \circ \ell_{\eta'}) \sim (H \circ \gamma'_{\kappa(s,t)(1)}),$$
whence
	$$(\chi_{\gamma,\kappa(s,t)} \circ \ell_{\eta'}) \sim (H \circ \gamma'_{\kappa(s,t)(0)})^{-1} \ast (\chi_{\gamma,\kappa(s,t)} \circ \ell_{\eta}) \ast (H \circ \gamma'_{\kappa(s,t)(1)}).$$
	But, from the first part of the lemma, $Im(H \circ \gamma'_{\kappa(s,t)(0)}) \subseteq V$ and $Im(H \circ \gamma'_{\kappa(s,t)(1)}) \subseteq W$, therefore
	$\tilde{\chi}_{\gamma,\kappa(s,t)}(\eta')=[\chi_{\gamma,\kappa(s,t)} \circ \ell_{\eta'}]\in N([\chi_{\gamma,\kappa(s,t)} \circ \ell_{\eta}], \mathcal{U}, V, W)$.
\end{proof}

\begin{lemma} \label{v-inv}
For every path $\gamma$, and every $s,t, \eta\in I$, $\chi_{\gamma,\kappa(s,t)} \circ \ell_{\eta}=(\chi_{\gamma,\kappa(t,s)} \circ \ell_{\eta})^{-1}$.
\end{lemma}
\begin{proof}
For every $x \in I$ we have that
\begin{align*}
(\chi_{\gamma,\kappa(t,s)} \circ \ell_{\eta})^{-1}(x)&= (\chi_{\gamma,\kappa(t,s)} \circ \ell_{\eta})(1-x)\\
&=H(\kappa(t,s)(1-x), \gamma(\eta))\\
&=H(\kappa(s,t)(x), \gamma(\eta))\\
&=(\chi_{\gamma,\kappa(s,t)} \circ \ell_{\eta})(x),
\end{align*}
therefore we have the equality.
\end{proof}

\begin{lemma}
Let $Y$ be any space and $\pi^{l}(Y)$ its fundamental groupoid equipped with Lasso topology. The map $\boldsymbol{\iota}: \pi^{l}(Y) \rightarrow \pi^{l}(Y)$ defined by $[p] \mapsto [p^{-1}]$ is continuous.
\end{lemma}
\begin{proof}
First we show that $\boldsymbol{\iota}$ is well defined. Indeed, if $p \sim q$, then $p \ast q^{-1} \sim c_{q(0)}=c_{p(0)}$ and $p^{-1} \ast p \ast q^{-1} \sim p^{-1} \ast c_{p(0)}$. But $p^{-1} \ast p \sim c_{p(1)}=c_{q^{-1}(0)}$, therefore $p^{-1} \sim q^{-1}$. Now we prove the continuity of $\boldsymbol{\iota}$. For every $[p] \in \pi^{l}(Y)$, let $N([p^{-1}], \mathcal{U},V,W])$ be an open set containing $[p^{-1}]=\boldsymbol{\iota}([p])$. Then $p^{-1}(0) \in V$ and $p^{-1}(1)\in W$. It follows from this that $N([p], \mathcal{U},W, V)$ is an open neighbourhood of $[p]$. For every $[q] \in N([p], \mathcal{U},W, V)$, $q(0) \in W$ and $q(1) \in V$, hence $q^{-1}(0)=q(1) \in V$ and $q^{-1}(1)=q(0) \in W$. Since $[q] \in N([p], \mathcal{U},W, V)$, then 
$$q \sim \lambda \ast \mu_{1} \ast p \ast \mu_{2} \ast \rho,$$
where $\lambda$ is a path in $W$ with $\lambda(1)=p(0)$, $\mu_{1}\in \pi_{1}(\mathcal{U}, p(0))$, $\mu_{2}\in \pi_{1}(\mathcal{U}, p(1))$ and $\rho$ is a path in $V$ with $\rho(0)=p(1)$. These conditions imply
$$q^{-1} \sim \rho^{-1} \ast \mu^{-1}_{2} \ast p^{-1} \ast \mu^{-1}_{1} \ast \lambda^{-1},$$
where $\rho^{-1}$ is a path in $V$ with $\rho^{-1}(1)=\rho(0)=p(1)=p^{-1}(0)$,
$\mu^{-1}_{2}\in \pi_{1}(\mathcal{U}, p^{-1}(0))$, $\mu^{-1}_{1}\in \pi_{1}(\mathcal{U}, p^{-1}(1))$ and $\lambda^{-1}$ is a path in $W$ with $\lambda^{-1}(0)=\lambda(1)=p(0)=p^{-1}(1)$. From these we get that $[q^{-1}] \in N([p^{-1}], \mathcal{U}, V, W)$ proving the continuity of $\boldsymbol{\iota}$.
\end{proof}

\begin{lemma}
The map $\boldsymbol{\iota}: \pi^{l}(X \times J) \rightarrow \pi^{l}(X \times J)$ induces a covariant functor $\pi(\boldsymbol{\iota}): \pi(\pi^{l}(X \times J)) \rightarrow \pi(\pi^{l}(X \times J))$ of groupoids. 
\end{lemma}
\begin{proof}
For every object $[\gamma] \in \pi^{l}(X \times J)$ we define 
$\pi(\boldsymbol{\iota})([\gamma])=[\gamma^{-1}]$ and for every homotopy class $[p]$ of a path $p$ in $\pi^{l}(X \times J)$, we let $\pi(\boldsymbol{\iota})([p])=[\boldsymbol{\iota} \circ p]$. Letting $e_{[\gamma]}=[c_{[\gamma]}]$ be the unit at $[\gamma]$, we see that
$$\pi(\boldsymbol{\iota})(e_{[\gamma]})=\pi(\boldsymbol{\iota})([c_{[\gamma]}])=[\boldsymbol{\iota}\circ c_{[\gamma]}]=[c_{[\gamma^{-1}]}]=e_{\pi(\boldsymbol{\iota})([\gamma])}.$$
For every two composable morphisms $[p],[q] \in \pi(\pi^{l}(X \times J))$ we see that
\begin{align*}
\pi(\boldsymbol{\iota})([p] \ast [q])&=\pi(\boldsymbol{\iota})([p \ast q])= [\boldsymbol{\iota} \circ (p \ast q)]=[(\boldsymbol{\iota} \circ p) \ast (\boldsymbol{\iota}\circ q)]\\
&=[\boldsymbol{\iota} \circ p] \ast [\boldsymbol{\iota} \circ q]= \pi(\boldsymbol{\iota})([p]) \ast \pi(\boldsymbol{\iota})([q]).
\end{align*}
proving the functoriality.
\end{proof}

\begin{definition}
Let $X$ be the discrete groupoid and $X \times \pi^{l}(X \times J)$ the obvious direct product groupoid. For every fuzzy subset $F$ we define $[F]: X \times \pi^{l}(X \times J) \rightarrow \pi(\pi^{l}(X \times J))$ on objects by $[F](y, (z,\beta))=[p_{F,(y,(z,\beta))}]$ where $p_{F,(y,(z,\beta))}$ is the path in $X \times J$ given by $p_{F,(y,(z,\beta))}(u)= H(\kappa(F(y),1-F(y))(u), (z,\beta))$, and on morphisms by $[F](y, [\gamma])=[\tilde{\chi}_{\gamma, \kappa( F(y),1-F(y))}]$.
\end{definition}

\begin{lemma} \label{constant}
	For every fuzzy subset $G$ and every $(y, [c_{(z,\beta)}]) \in X \times \pi^{l}(X \times J)$, the path $\tilde{\chi}_{c_{(z,\beta)}, \kappa(G(y), 1-G(y))}$ is the constant path $c_{[p_{G,(y,(z,\beta))}]}$ at $[p_{G,(y,(z,\beta))}]$.
\end{lemma}
\begin{proof}
	For every $\eta \in I$, 
	\begin{equation} \label{eq-const}
		(\tilde{\chi}_{c_{(z,\beta)}, \kappa(G(y), 1-G(y))})(\eta)=[\chi_{c_{(z,\beta)}, \kappa(G(y),1-G(y))} \circ \ell_{\eta}].
	\end{equation} 
	To prove the claim we show that paths $\chi_{c_{(z,\beta)}, \kappa(G(y),1-G(y))} \circ \ell_{\eta}$ in $X \times J$ with $\eta$ varying in $I$ coincide. Indeed, for every $u \in I$, 
	\begin{align*}
		(\chi_{c_{(z,\beta)}, \kappa(G(y),1-G(y))} \circ \ell_{\eta})(u)&= H(\kappa(G(y),1-G(y))(u), c_{(z,\beta)}(\eta))\\
		&=H(\kappa(G(y),1-G(y))(u), (z,\beta)), \stepcounter{equation}\tag{\theequation}\label{H}
	\end{align*}
	where the right hand side is independent of $\eta$. The second part of the lemma follows from (\ref{eq-const}) and (\ref{H}).
\end{proof}

\begin{proposition}
For every fuzzy subset $F$, $[F]$ is a covariant functor of groupoids.
\end{proposition}
\begin{proof}
First we prove the correctness of $[F]$. Let $\gamma \sim \delta$ and we want to show that for every $y \in X$, $\tilde{\chi}_{\gamma, \kappa(F(y), 1-F(y))} \sim \tilde{\chi}_{\delta, \kappa(F(y), 1-F(y))}$. If $\varphi:I \times I\rightarrow X \times J$ is a homotopy between $\gamma$ and $\delta$, then for every $t \in I$, the restriction $\varphi(\bullet,t)$ of $\varphi$ on $I \times \{t\}$ is a path in $X \times J$ from $\gamma(0)=\delta(0)$ to $\gamma(1)=\delta(1)$. We let $\tilde{\varphi}: I \times I \rightarrow \pi^{l}(X \times J)$ be the map given by
$$\tilde{\varphi}(\eta, t)=\tilde{\chi}_{\varphi(\bullet,t), F(y), 1-F(y)}(\eta)=[\chi_{\varphi(\bullet, t), \kappa(F(y), 1-F(y))} \circ \ell_{\eta}].$$
This map is a homotopy between $\tilde{\chi}_{\gamma, \kappa(F(y), 1-F(y))}$ and $\tilde{\chi}_{\delta, \kappa(F(y), 1-F(y))}$. Indeed, 
$$\tilde{\varphi}(\eta, 0)=[\chi_{\varphi(\bullet, 0), \kappa(F(y), 1-F(y))} \circ \ell_{\eta}]=[\chi_{\gamma, \kappa(F(y), 1-F(y))} \circ \ell_{\eta}]=(\tilde{\chi}_{\gamma, \kappa(F(y), 1-F(y))})(\eta),$$
and 
$$\tilde{\varphi}(\eta, 1)=[\chi_{\varphi(\bullet, 1), \kappa(F(y), 1-F(y))} \circ \ell_{\eta}]=[\chi_{\delta, \kappa(F(y), 1-F(y))} \circ \ell_{\eta}]=(\tilde{\chi}_{\delta, \kappa(F(y), 1-F(y))})(\eta).$$
Also for every $t \in I$,
$$\tilde{\varphi}(0, t)= [\chi_{\varphi(\bullet, t), \kappa(F(y), 1-F(y))} \circ \ell_{0}]=[\chi_{\gamma, \kappa(F(y), 1-F(y))} \circ \ell_{0}]=(\tilde{\chi}_{\gamma, \kappa(F(y),1-F(y))})(0),$$
where the second equality holds true since for every $u \in I$,
\begin{align*}
(\chi_{\varphi(\bullet, t), \kappa(F(y), 1-F(y))} \circ \ell_{0})(u)&=H(\kappa(F(y), 1-F(y))(u), \varphi(0,t))\\
&=H(\kappa(F(y), 1-F(y))(u), \gamma(0))\\
&=(\chi_{\gamma, \kappa(F(y), 1-F(y))} \circ \ell_{0})(u).
\end{align*}
In a similar way one can check that
$$\tilde{\varphi}(1, t)=(\tilde{\chi}_{\gamma, \kappa(F(y),1-F(y))})(1).$$
It remains now to prove the continuity of $\tilde{\varphi}$. Let $(\eta, t) \in I \times I$ be arbitrary and $N([\chi_{\varphi(\bullet, t), \kappa(F(y), 1-F(y))} \circ \ell_{\eta}], \mathcal{U}, V, W)$ be any open neighbourhood of $[\chi_{\varphi(\bullet, t), \kappa(F(y), 1-F(y))} \circ \ell_{\eta}]$ in $\pi^{l}(X \times J)$. It follows that $V$ is an open neighbourhood of 
$$(\chi_{\varphi(\bullet, t), \kappa(F(y), 1-F(y))} \circ \ell_{\eta})(0)=H(F(y), \varphi(\eta, t)),$$
and that $W$ is an open neighbourhood of 
$$(\chi_{\varphi(\bullet, t), \kappa(F(y), 1-F(y))} \circ \ell_{\eta})(1)=H(1-F(y), \varphi(\eta, t)).$$
The continuity of $H$ implies that there is an open $A$ in $I$ and an open $B$ in $X \times J$ such that $F(y) \in A$, $\varphi(\eta, t) \in B$ and that $A \times B \subseteq H^{-1}(V)$. Similarly, there is an open $A'$ in $I$ and an open $B'$ in $X \times J$ such that $1-F(y) \in A'$, $\varphi(\eta, t) \in B'$ and that $A' \times B' \subseteq H^{-1}(W)$. Consequently $B \cap B'$ is an open neighbourhood of $\varphi(\eta, t)$ in $X \times J$. The continuity of $\varphi$ implies that $\varphi^{-1}(B \cap B')$ is an open neighbourhood of $(\eta, t)$ in $I \times I$. We can chose some $r>0$ such that the open ball $B((\eta, t), r)$ with center $(\eta, t)$ and radius $r$ is included in $\varphi^{-1}(B \cap B')$. We prove that $\tilde{\varphi}(B((\eta, t), r)) \subseteq N([\chi_{\varphi(\bullet, t), \kappa(F(y), 1-F(y))} \circ \ell_{\eta}], \mathcal{U}, V, W)$ which implies the continuity of $\tilde{\varphi}$. Let $(\eta', t') \in B((\eta, t), r)$ and want to prove that 
\begin{equation} \label{tf}
\tilde{\varphi}(\eta', t') = [\chi_{\varphi(\bullet, t'), \kappa(F(y), 1-F(y))} \circ \ell_{\eta'}] \in N([\chi_{\varphi(\bullet, t), \kappa(F(y), 1-F(y))} \circ \ell_{\eta}], \mathcal{U}, V, W).
\end{equation}
The homotopy $\varphi$ induces a path $\varphi(\eta', \bullet): I \rightarrow X \times J$ by the rule $t'' \mapsto \varphi(\eta', t'')$ which in turn induces two paths, $H \circ \varphi(\eta', \bullet)_{F(y)}$ and $H \circ \varphi(\eta', \bullet)_{1-F(y)}$ which are freely homotopic. Under the assumption that $t'<t$, Lemma \ref{path-res} implies that the restrictions $\alpha$ and $\alpha'$ of respectively $H \circ \varphi(\eta', \bullet)_{F(y)}$ and $H \circ \varphi(\eta', \bullet)_{1-F(y)}$ on $[t',t]$ are also freely homotopic, and the path which connects $H(F(y), \varphi(\eta',t'))$ with $H(1-F(y), \varphi(\eta',t'))$ is $\chi_{\varphi(\eta', \bullet), \kappa(F(y), 1-F(y))} \circ \ell_{t'}$, and the path which connects $H(F(y), \varphi(\eta', t))$ with $H(1-F(y), \varphi(\eta', t))$ is $\chi_{\varphi(\eta', \bullet), \kappa(F(y), 1-F(y))} \circ \ell_{t}$. It is easy to see that  
$$\chi_{\varphi(\eta', \bullet), \kappa(F(y), 1-F(y))} \circ \ell_{t'}= \chi_{\varphi(\bullet, t'), \kappa(F(y), 1-F(y))} \circ \ell_{\eta'},$$
and 
$$\chi_{\varphi(\eta', \bullet), \kappa(F(y), 1-F(y))} \circ \ell_{t}= \chi_{\varphi(\bullet, t), \kappa(F(y), 1-F(y))} \circ \ell_{\eta'}.$$
It follows from Lemma \ref{relative homotopy} that
\begin{equation} \label{v}
\chi_{\varphi(\bullet, t), \kappa(F(y), 1-F(y))} \circ \ell_{\eta'} \sim \alpha^{-1} \ast (\chi_{\varphi(\bullet, t'), \kappa(F(y), 1-F(y))} \circ \ell_{\eta'}) \ast \alpha'.
\end{equation}
We prove that $\alpha$ is included in $V$ and that $\alpha'$ is included in $W$. We present here the proof for the first claim since the proof for the second is dual. Any point of $\alpha$ is of the form $H(F(y), \varphi(\eta', t''))$ with $t'' \in [t', t]$ where $(\eta', t'') \in B((\eta,t), r)$ since
$$d((\eta', t''), (\eta, t))=\sqrt{(\eta-\eta')^{2} + (t-t'')^{2}} \leq \sqrt{(\eta-\eta')^{2}+ (t-t')^{2}}< r.$$
This implies that 
$$\varphi(\eta', t'') \in \varphi(B((\eta,t), r)) \subseteq \varphi(\varphi^{-1}(B \cap B')) \subseteq \varphi(\varphi^{-1}(B))=B.$$
But $F(y) \in A$, therefore $(F(y), \varphi(\eta', t'')) \in A \times B \subseteq H^{-1}(V)$, and then $H(F(y), \varphi(\eta', t'')) \in H(H^{-1}(V))=V$. Now the facts that $\alpha^{-1} \subseteq V$ and that $\alpha' \subseteq W$ together with (\ref{v}) imply that
\begin{equation} \label{vh}
[\chi_{\varphi(\bullet, t), \kappa(F(y), 1-F(y))} \circ \ell_{\eta'} ] \in N([\chi_{\varphi(\bullet, t'), \kappa(F(y), 1-F(y))} \circ \ell_{\eta'}], \mathcal{U}, V, W).
\end{equation}
Further, under the assumption that $\eta'<\eta$, Lemma \ref{path-res} implies that the restrictions $\beta$ and $\beta'$ of respectively $H \circ \varphi(\bullet, t)_{F(y)}$ and $H \circ \varphi(\bullet, t)_{1-F(y)}$ on $[\eta', \eta]$ are homotopic with each other, and the path which connects $H(F(y), \varphi(\eta',t))$ with $H(1-F(y), \varphi(\eta',t))$ is $\chi_{\varphi(\bullet, t), \kappa(F(y), 1-F(y))} \circ \ell_{\eta'}$, and the path which connects $H(F(y), \varphi(\eta, t))$ with $H(1-F(y), \varphi(\eta, t))$ is $\chi_{\varphi(\bullet, t), \kappa(F(y), 1-F(y))} \circ \ell_{\eta}$. Since
$$(\chi_{\varphi(\bullet, t), \kappa(F(y), 1-F(y))} \circ \ell_{\eta'})^{-1} \ast \beta \ast (\chi_{\varphi(\bullet, t), \kappa(F(y), 1-F(y))} \circ \ell_{\eta}) \sim \beta',$$
it follows that
\begin{equation} \label{h}
\beta \ast (\chi_{\varphi(\bullet, t), \kappa(F(y), 1-F(y))} \circ \ell_{\eta}) \ast \beta'^{-1} \sim \chi_{\varphi(\bullet, t), \kappa(F(y), 1-F(y))} \circ \ell_{\eta'}.
\end{equation}
Also we have that $\beta$ is included in $V$ and $\beta'$ is included in $W$. We present here the proof for $\beta'$ since the proof for $\beta$ is dual and similar to the proof that $\alpha$ is included in $V$. Any point on $\beta'$ has the form $H(1-F(y), \varphi(\eta'', t))$ with $\eta'' \in [\eta', \eta]$ and $(\eta'', t) \in B((\eta, t),r)$ since
$$d((\eta'', t), (\eta, t))=\sqrt{(\eta-\eta'')^{2}+(t-t)^{2}}\leq \sqrt{(\eta-\eta')^{2}+(t-t')^{2}}<r.$$
This implies that
$$\varphi(\eta'', t) \in \varphi (B((\eta, t),r)) \subseteq \varphi (\varphi^{-1}(B\cap B'))\subseteq \varphi (\varphi^{-1}(B'))=B'.$$
This, together with $1-F(y) \in A'$, imply that $(1-F(y), \varphi(\eta'', t)) \in A' \times B' \subseteq H^{-1}(W)$, consequently $H(1-F(y), \varphi(\eta'', t)) \in H(H^{-1}(W))=W$. Now (\ref{h}) and the facts that $\beta \subseteq V$ and $\beta' \subseteq W$ imply that 
\begin{equation} \label{hh}
[\chi_{\varphi(\bullet, t), \kappa(F(y), 1-F(y))} \circ \ell_{\eta'}] \in N([\chi_{\varphi(\bullet, t), \kappa(F(y), 1-F(y))} \circ \ell_{\eta}], \mathcal{U}, V, W).
\end{equation}
Now (\ref{vh}), (\ref{hh}) and Lemma 2.3 of \cite{LT} imply (\ref{tf}) as desired.

To check the functoriality of $[F]$ we first see that $[F]$ sends identities to identities. Indeed, for every object $(y, (z,\beta))$,
\begin{align*}
[F](e_{(y, (z,\beta))})&=[F](y, [c_{(z,\beta)}])\\
&=[\tilde{\chi}_{c_{(z,\beta)}, \kappa( F(y),1-F(y))}]\\
&=[c_{[p_{F,(y,(z,\beta))}]}] && (\text{by Lemma \ref{constant}})\\
&=[c_{[F](y, (z,\beta))}] && (\text{by the definition})\\
&=e_{[F](y, (z,\beta))}
\end{align*} 
Let now $[\gamma], [\delta] \in \pi^{l}(X \times J)$ such that $[\gamma] \ast [\delta]$ exists, then for every $y \in X$, $(y, [\gamma]) \ast (y, [\delta])=(y, [\gamma] \ast [\delta])=(y, [\gamma \ast \delta])$. Let us check that
$[F](y,[\gamma])=[\tilde{\chi}_{\gamma, \kappa(F(y),1-F(y))}]$ is composable with $[F](y,[\delta])=[\tilde{\chi}_{\delta, \kappa(F(y),1-F(y))}]$. This follows if we prove that $(\tilde{\chi}_{\gamma, \kappa(F(y),1-F(y))})(1)=(\tilde{\chi}_{\delta, \kappa(F(y),1-F(y))})(0)$, or equivalently that $\chi_{\gamma, \kappa(F(y),1-F(y))} \circ \ell_{1}=\chi_{\delta, \kappa(F(y),1-F(y))} \circ \ell_{0}$. This is indeed so since for every $x \in I$ we have that
\begin{align*}
(\chi_{\gamma, \kappa(F(y),1-F(y))} \circ \ell_{1})(x)&=H(\kappa(F(y),1-F(y))(x), \gamma(1))\\
&=H(\kappa(F(y),1-F(y))(x), \delta(0)) && (\text{$\gamma \ast \delta$ exists})\\
&= (\chi_{\delta, \kappa(F(y),1-F(y))} \circ \ell_{0})(x).
\end{align*}
Now we have to prove that 
$$[F](y, [\gamma]) \ast [F](y, [\delta])=[F](y, [\gamma \ast \delta]),$$
or equivalently that
$$[\tilde{\chi}_{\gamma, \kappa(F(y),1-F(y))}] \ast [\tilde{\chi}_{\delta, \kappa(F(y),1-F(y))}]=[\tilde{\chi}_{\gamma \ast \delta, \kappa(F(y),1-F(y))}].$$
We prove this by proving that
$$(\tilde{\chi}_{\gamma, \kappa(F(y),1-F(y))}) \ast (\tilde{\chi}_{\delta, \kappa(F(y),1-F(y))})=\tilde{\chi}_{(\gamma \ast \delta), \kappa(F(y),1-F(y))}.$$
The path in the left hand side of the above is given by
$$(\tilde{\chi}_{\gamma, \kappa(F(y),1-F(y))} \ast \tilde{\chi}_{\delta, \kappa(F(y),1-F(y))})(\eta)=\left\lbrace \begin{matrix}
	[\chi_{\gamma, \kappa(F(y),1-F(y))} \circ \ell_{2\eta}] & if & 0 \leq \eta \leq 1/2\\
	[\chi_{\delta, \kappa(F(y),1-F(y))} \circ \ell_{2\eta-1}] & if & 1/2 \leq \eta \leq 1
\end{matrix} \right.$$
Observe that for every $x \in I$, if $0\leq \eta \leq 1/2$, then
$$(\chi_{\gamma, \kappa(F(y),1-F(y))} \circ \ell_{2\eta})(x)=H(\kappa(F(y),1-F(y))(x), \gamma (2\eta)),$$
and when $1/2 \leq \eta \leq 1$, then
$$(\chi_{\delta, \kappa(F(y),1-F(y))} \circ \ell_{2\eta-1})(x)=H(\kappa(F(y),1-F(y))(x), \delta(2\eta -1)).$$
On the other hand, since
$$(\tilde{\chi}_{(\gamma \ast \delta), \kappa(F(y),1-F(y))})(\eta)=[\chi_{(\gamma \ast \delta), \kappa(F(y),1-F(y))} \circ \ell_{\eta}],$$
we have that for every $x \in I$,
$$(\chi_{(\gamma \ast \delta), \kappa(F(y),1-F(y))} \circ \ell_{\eta})(x)=H(\kappa(F(y),1-F(y))(x), (\gamma \ast \delta)(\eta)).$$
It follows that for $0 \leq \eta \leq 1/2$,
$$(\chi_{(\gamma \ast \delta), \kappa(F(y),1-F(y))} \circ \ell_{\eta})(x)=H(\kappa(F(y),1-F(y))(x), \gamma (2\eta)),$$
and for $1/2 \leq \eta \leq 1$,
$$(\chi_{(\gamma \ast \delta), \kappa(F(y),1-F(y))} \circ \ell_{\eta})(x)=H(\kappa(F(y),1-F(y))(x), \delta (2\eta-1)).$$
By comparing we get the desired equality.
\end{proof}

\begin{theorem} \label{main}
\begin{itemize}
	\item [(i)] For every two fuzzy subsets $F$ and $G$, $\pi(\boldsymbol{\iota}) \circ [F]=[G]$ if and only if $G=1-F$. 
	\item[(ii)] For every fuzzy subset $F$, for every $y \in X$ and every constant path $c_{(z,\beta)}$ at some $(z,\beta)$ we have in particular that $[1-F](x, [c_{(z,\beta)}])=([F](x, [c_{(z,\beta)}]))^{-1}$.
\end{itemize}
\end{theorem}
\begin{proof}
(i) Assume that $G=1-F$ and let $(y, [\gamma]) \in X \times \pi^{l}(X \times J)$ be arbitrary. Then,
$$(\pi(\boldsymbol{\iota}) \circ [F])(y, [\gamma])=\pi(\boldsymbol{\iota})([\tilde{\chi}_{\gamma, \kappa(F(y),1-F(y))}])=[\boldsymbol{\iota} \circ \tilde{\chi}_{\gamma, \kappa(F(y),1-F(y))}],$$
and $[1-F](y, [\gamma])=[\tilde{\chi}_{\gamma, \kappa(1-F(y),F(y))}]$. The claim follows if prove that for every $\eta \in I$,
$$(\boldsymbol{\iota} \circ \tilde{\chi}_{\gamma, \kappa(F(y),1-F(y))})(\eta)=(\tilde{\chi}_{\gamma, \kappa(1-F(y),F(y))})(\eta).$$
But on the one side
$$(\boldsymbol{\iota} \circ \tilde{\chi}_{\gamma, \kappa(F(y),1-F(y))})(\eta)=\boldsymbol{\iota}([\chi_{\gamma , \kappa(F(y),1-F(y))} \circ \ell_{\eta}])=[(\chi_{\gamma , \kappa(F(y),1-F(y))} \circ \ell_{\eta})^{-1}],$$
and on the other side
$$(\tilde{\chi}_{\gamma, \kappa(1-F(y),F(y))})(\eta)=[\chi_{\gamma , \kappa(1-F(y),F(y))} \circ \ell_{\eta}].$$
From Lemma \ref{v-inv} we have that $(\chi_{\gamma , \kappa(F(y),1-F(y))} \circ \ell_{\eta})^{-1}=\chi_{\gamma , \kappa(1-F(y),F(y))} \circ \ell_{\eta}$, therefore the equality follows.

Conversely, assume that $\pi(\boldsymbol{\iota}) \circ [F]=[G]$, therefore for every $(y, [\gamma]) \in X \times \pi^{l}(X \times J)$ we have that
$$[\tilde{\chi}_{\gamma, \kappa(G(y),1-G(y))}]=[G](y, [\gamma])=(\pi(\boldsymbol{\iota}) \circ [F])(y, [\gamma])=[\boldsymbol{\iota} \circ \tilde{\chi}_{\gamma, \kappa(F(y),1-F(y))}].$$
This implies that for $\eta=0$ in particular we have,
$$[\chi_{\gamma , \kappa(G(y),1-G(y))} \circ \ell_{0}]=(\tilde{\chi}_{\gamma, \kappa(G(y),1-G(y))})(0)=(\boldsymbol{\iota} \circ \tilde{\chi}_{\gamma, \kappa(F(y),1-F(y))})(0)=[(\chi_{\gamma , \kappa(F(y),1-F(y))} \circ \ell_{0})^{-1}].$$
Taking $x=0$ we obtain
\begin{align*}
H(G(y), \gamma(0))&=H(\kappa(G(y),1-G(y))(0), \gamma(0))\\
&=(\chi_{\gamma , \kappa(G(y),1-G(y))} \circ \ell_{0})(0)\\
&= (\chi_{\gamma , \kappa(F(y),1-F(y))} \circ \ell_{0})^{-1}(0)\\
&= (\chi_{\gamma , \kappa(1-F(y),F(y))} \circ \ell_{0})(0) && (\text{from Lemma \ref{v-inv}})\\
&= H(1-F(y), \gamma(0)).
\end{align*}
Since this holds true for every path $\gamma$, we can chose it to satisfy $\gamma(0)=(z,\beta)$ with $\beta \neq 0$, hence we have $(1-G(y))\beta=F(y)\beta$. This implies that for every $y \in X$, $G(y)=1-F(y)$ therefore $G=1-F$.\\
(ii) Combining part (i) of the theorem together with Lemma \ref{constant} we have the following
\begin{align*}
[1-F](y, [c_{(z,\beta)}])&=\pi(\boldsymbol{\iota})([F](y, [c_{(z,\beta)}])) && (\text{by part (i) of the theorem})\\
&=\pi(\boldsymbol{\iota})([\tilde{\chi}_{c_{(z,\beta)}, \kappa(F(y), 1-F(y))}])\\
&= \pi(\boldsymbol{\iota})[c_{[p_{F, (y,(z, \beta))}]}] && (\text{by Lemma \ref{constant}})\\
&= [c_{[p_{F, (y,(z, \beta))}]}]^{-1} && (\text{definition of $\pi(\boldsymbol{\iota})$})\\
&= ([F](y, [c_{(z,\beta)}]))^{-1} && (\text{by Lemma \ref{constant}}.)
\end{align*}
This completes the proof.
\end{proof}

\end{document}